\documentclass[12pt, onside]{amsart}
\usepackage[margin=1in]{geometry}

\usepackage[T1]{fontenc}
\usepackage{amsmath,amssymb,mathtools,amsthm}
\usepackage{enumitem}
\usepackage{microtype}
\usepackage{tikz}
\usetikzlibrary{arrows.meta}
\usepackage[hidelinks]{hyperref}

\newcommand{\R}{\mathbb{R}}
\newcommand{\N}{\mathbb{N}}
\newcommand{\Nzero}{\mathbb{N}_0}
\newcommand{\eps}{\varepsilon}
\DeclareMathOperator{\diam}{diam}
\DeclareMathOperator{\conv}{conv}

\theoremstyle{plain}
\newtheorem{theorem}{Theorem}[section]
\newtheorem{proposition}[theorem]{Proposition}
\newtheorem{lemma}[theorem]{Lemma}
\newtheorem{corollary}[theorem]{Corollary}
\theoremstyle{definition}
\newtheorem{definition}[theorem]{Definition}
\theoremstyle{remark}
\newtheorem{remark}[theorem]{Remark}
\newtheorem{claim}[theorem]{Claim}

\newcommand{\lp}{\left(}
\newcommand{\rp}{\right)}
\newcommand{\lb}{\left[}
\newcommand{\rb}{\right]}

\numberwithin{equation}{section}

\title[Arithmetic-progression gap sets]{Arithmetic-progression gap sets in Cantor sets}

\author{S. Sandberg-Clark}
\address{Department of Mathematical Sciences, United States Military Academy at West Point}
\email{samantha.sandbergclark@westpoint.edu}

\author{K. Taylor}
\address{Department of Mathematics, The Ohio State University}
\email{taylor.2952@osu.edu}

\author{A. Yavicoli}
\address{Department of Mathematics, The University of British Columbia}
\email{yavicoli@math.ubc.ca}

\date{}

\begin{document}

\begin{abstract}
We address the question of which common differences can arise in arithmetic progressions contained in fractal sets. For a compact set $C\subset\mathbb R$, we investigate not only whether arithmetic progressions occur in $C$, but the full collection of their common differences. More generally, for a finite pattern $P$, we study the set of scales at which affine copies of $P$ appear in $C$.

For affine self-similar sets satisfying strong separation, 
we obtain explicit restrictions on admissible common differences.
 Specializing to middle-$\varepsilon$ Cantor sets, we prove that the longest arithmetic progression has length four whenever $3-2\sqrt2<\varepsilon\le 1/3$, showing that the maximal progression length drops immediately from six at the critical parameter $\varepsilon=3-2\sqrt2$.
We further develop recursive bounds for the sets of admissible common differences and derive explicit blackout intervals, namely ranges of scales for which arithmetic progressions cannot occur.

On the positive side, sufficiently thick Cantor sets exhibit the opposite behavior. Combining a refinement of the Hunt-Kan-Yorke construction with the Newhouse Gap Lemma, we prove that every sufficiently small common difference occurs in a three-term arithmetic progression. 
In particular, if the largest bounded gap of a Cantor set is at most $0.067\,\diam(C)$ and its thickness is at least $6.96268\ldots$, then every common difference in $(0,0.435\,\diam(C)]$ occurs in a three-term arithmetic progression contained in $C$.
Analogous interval results are obtained for four-term arithmetic progressions and asymmetric three-point patterns.
\end{abstract}


\maketitle

The occurrence of finite configurations in large or structured sets is a
central theme in additive combinatorics and fractal geometry.  In the
integer setting, largeness is often expressed through density.  For compact
subsets of Euclidean space, Hausdorff dimension, Fourier decay, and Newhouse
thickness provide different quantitative notions of size.  The present paper
asks not only whether a compact set contains a prescribed configuration, but
also at which scales that configuration occurs.

For a compact set $C\subset\R$ and an integer $k\ge2$, define the 
\emph{arithmetic-progression gap set}
\begin{equation}\label{eq:def-gap-set}
G^k_{AP}(C)
:=\bigl\{t>0:\ \exists x\, \in\R\text{ with }
 x,x+t,\ldots,x+(k-1)t\in C\bigr\},
\end{equation}
where we refer to $t\in G^k_{AP}(C)$ as a \textit{common difference}. 
More generally, let
$$
P=\{0=p_1<p_2<\cdots<p_k=1\}\subset[0,1]
$$
be a normalized finite pattern.  
Its \emph{scale set} in $C$ is
\begin{equation}\label{eq:def-pattern-scale}
G_P(C):=\{s>0:\ \exists x\, \in\R\text{ such that }x+sP\subset C\}.
\end{equation}
Here $s$ is the diameter of the embedded copy.  We write
$$
\delta_j(P)=p_j-p_{j-1},\qquad
\delta_{\min}(P)=\min_{2\le j\le k}\delta_j(P),\qquad
\delta_{\max}(P)=\max_{2\le j\le k}\delta_j(P).
$$
For
$$
\mathcal A_k=\left\{0,\frac1{k-1},\ldots,1\right\},
$$
the two conventions satisfy
\begin{equation}\label{eq:AP-pattern-conversion}
G_{\mathcal A_k}(C)=(k-1)G^k_{AP}(C).
\end{equation}
For asymmetric three-point patterns we use
$$
P_\theta=\{0,\theta,1\},\qquad 0<\theta<1.
$$

The scale-set problem is an affine-configuration analogue of the distance-set
problem.  Falconer's conjecture asks whether $\dim_H E>d/2$ forces the
distance set $\Delta(E)=\{|x-y|:x,y\in E\}$ to have positive Lebesgue
measure \cite{FalconerDistance1985}; Mattila and Sj\"olin proved nonempty
interior under the stronger condition $\dim_H E>(d+1)/2$
\cite{MattilaSjolin1999}.  Arithmetic progressions impose a simultaneous
multipoint relation, and dimensional information alone does not determine
their existence.  Keleti constructed full-dimensional compact subsets of
$\R$ avoiding prescribed three-point similarity patterns
\cite{keleti2008construction}.  Yavicoli strengthened this to positive
measure for gauges arbitrarily close to the ambient-dimensional gauge while
avoiding countably many linear patterns \cite{Yavicoli_Avoiding}.  Shmerkin
constructed full-dimensional Salem sets without nontrivial three-term
arithmetic progressions \cite{Shmerkin_2016}.  Conversely, sufficiently
strong dimensional and Fourier-decay estimates do force three-term
progressions \cite{laba_2009}.  More recently, Carnovale and Senger studied
Lebesgue- and Hausdorff-size conclusions for the set of three-term
progression step sizes under Fourier-analytic hypotheses
\cite{CarnovaleSenger}.  Our emphasis is complementary: we seek exact or
recursive information for specific Cantor sets, explicit blackout intervals,
and guaranteed intervals of scales arising from self-similarity and
thickness.  These results motivate the use of geometric structure, and
especially thickness, rather than dimension alone.

Arithmetic progressions in middle Cantor sets and in homogeneous
self-similar sets have also been studied directly.  Chaika proved that the
middle-$N$th Cantor set contains arithmetic progressions of length at least a
constant multiple of $N/\log_2 N$ \cite{ChaikaMiddleNth}.  Broderick,
Fishman, and Simmons obtained, for sufficiently small $\eps$, quantitative
bounds of the form
$$\frac{1/\eps}{\log(1/\eps)}\lesssim L_{AP}(C_\eps) \le \frac1\eps+1,$$
where $L_{AP}(E)$ denotes the maximal length of a nontrivial arithmetic
progression in $E$; they also recorded the Newhouse-Gap-Lemma consequence
$L_{AP}(C_\eps)\ge4$ for $0<\eps\le1/3$
\cite{BroderickFishmanSimmons}.  Xi, Jiang, and Pei investigated existence
and maximal-length questions for homogeneous self-similar sets using
multiple $\beta$-expansions and asked whether $L_{AP}(C_\eps)$ is an
increasing staircase function of the contraction ratio
\cite{XiJiangPei}.  
Tian, Lou, and Shang answered that question negatively; in particular, they proved the exact endpoint value
\begin{equation}\label{eq:TLS-endpoint}
 L_{AP}(C_{3-2\sqrt2})=6.
\end{equation}
Thus the strict lower endpoint in our maximal-length-four theorem is essential \cite{TianLouShang}.  
In contrast with this earlier maximal-length literature, the main objects here are the entire gap and scale sets $G^k_{AP}(C)$ and $G_P(C)$, including their recursive structure, blackout intervals, and intervals of admissible scales.

The middle-$\eps$ Cantor set $C_\eps$ is our principal model.  
It is generated by the iterated function system
$$f_0(x)=\lambda x,\qquad f_1(x)=1-\lambda+\lambda x,$$ where $\lambda=(1-\eps)/2$.
In the case where $\eps=1/3$, we get the middle-third Cantor set which illustrates the distinction between distances and progression gaps: $\Delta(C_{1/3})=[0,1]$, whereas
$$G^3_{AP}(C_{1/3})=G^4_{AP}(C_{1/3})
=\{3^{-n}:n\in\N\}.$$

\subsection*{Thickness, higher-order intersections, and the interior problem}
The interior question can be formulated as a uniform higher-order intersection problem.  
Indeed,
\begin{equation}\label{eq:intro-AP-intersection}
t\in G^k_{AP}(C)
\quad\Longleftrightarrow\quad
\bigcap_{j=0}^{k-1}(C-jt)\ne\varnothing,
\end{equation}
and, for a normalized pattern $P=\{p_1,\ldots,p_k\}$,
\begin{equation}\label{eq:intro-pattern-intersection}
s\in G_P(C)
\quad\Longleftrightarrow\quad
\bigcap_{j=1}^k(C-sp_j)\ne\varnothing.
\end{equation}
Thus proving that a scale set has interior requires a single geometric hypothesis to force these intersections for every parameter in an interval, not merely for one exceptional scale.

For two-set intersections, Newhouse thickness provides precisely such a robust mechanism.  
The Newhouse Gap Lemma \cite{newhouse} implies, for
example, that if $\tau(C)\ge1$, then $C\cap(C+t)\ne\varnothing$ for every $0<t<\diam(C)$; equivalently, the two-point scale set contains an interval.
The same lemma can be used more indirectly to produce an individual three-term arithmetic progression in every compact set of thickness at least one: one separates suitable pieces of the set and intersects the set with an appropriate midpoint set.  
See \cite[Proposition~20]{Yavicoli_Survey} and \cite{Sandberg_Taylor} for this idea and its extensions to general three-point configurations.  
Stronger thickness hypotheses also force larger prescribed finite patterns \cite{yavicoli_patterns}.  
These results, however, are primarily existence statements and do not by themselves give a uniform interval of admissible scales.

The theorem of Hunt, Kan, and Yorke gives a route from two-set to higher-order intersections: under suitable thickness assumptions, the intersection of two interleaved Cantor sets contains a compact subset of positive thickness \cite{hunt_kan_yorke}.  
For three-term progressions, the natural strategy is to construct
$$K_s\subseteq (C-s)\cap(C+s)$$ with quantitative lower bounds for both $\tau(K_s)$ and $\diam(K_s)$, and then apply the Gap Lemma once more to $K_s$ and $C$.  The qualitative HKY statement is not enough for an explicit interval of parameters, because the second application requires bounds that are uniform in $s$.  
In Section~\ref{sec:HKY section triple intersection} we extract such bounds directly from the initial stage of the HKY construction.  
This yields the explicit three-term interval in Theorem~\ref{thm:hky triple intersection}; a second Gap Lemma argument, applied to the HKY output and a short translate of it, gives the four-term interval in Theorem~\ref{thm:hky quadruple intersection}.

\subsection*{Main results}
The paper develops complementary exclusion and interior results.
\begin{enumerate}[leftmargin=2.4em,label=\textup{(\arabic*)}]
\item (Maximal AP length For $3-2\sqrt2<\eps\le1/3$, the longest arithmetic progression in $C_\eps$ has exactly four terms (Theorem~\ref{thm:maxAP4}).  Together with \eqref{eq:TLS-endpoint}, this shows that the maximal length drops from six at the endpoint to four immediately to its right.

\item (Blackout intervals) For affine self-similar sets satisfying strong separation, a general first-splitting theorem confines $G_P(C)$ to an explicit union of rescaled intervals (Theorem~\ref{thm:general-pattern-blackout}).  
A level-two refinement for $C_\eps$ lowers the threshold obtained from the first-level bound and proves that $G^3_{AP}(C_\eps)$ contains no interval $(0,r)$ whenever
$$\frac{4-\sqrt{13}}3<\eps\le\frac13.$$
The argument also gives explicit blackout intervals in the additional parameter range (see Theorem~\ref{thm:improved-middle-epsilon-blackout}).

\item (Interior of the $3$-AP gap set)
An explicit initial-stage lower bound extracted from the Hunt--Kan--Yorke construction implies that, if $C$ has diameter $d$, largest bounded gap at most $0.067d$, and thickness at least
$(1-0.067)/(2\cdot0.067)$, then
$$ (0,0.435d]\subseteq G^3_{AP}(C) $$
(see Theorem~\ref{thm:hky triple intersection}).

\item (Asymmetric three-point patterns) Using the same assumptions above, for every $0<\theta<1$ we prove both blackout criteria and positive intervals for $G_{P_\theta}(C)$.  
In particular, if $0<\eps\le0.067$, then 
$$(0,0.87]\subseteq G_{P_\theta}(C)$$
(see Corollary~\ref{cor:asymmetric patterns}).

\item (Interior of the $4$-AP gap set)
Combining the same two-set construction with a second application of
the Newhouse Gap Lemma gives
$$(0,0.30d]\subseteq G^4_{AP}(C)$$
when $\tau(C)\ge32.83333\ldots$
(see Theorem~\ref{thm:hky quadruple intersection}).
\end{enumerate}

The novelty and attribution of these results can be summarized as follows.
Prior work on middle and homogeneous self-similar Cantor sets concerns existence and maximal AP length \cite{ChaikaMiddleNth,BroderickFishmanSimmons,XiJiangPei,TianLouShang}.  
Our first two contributions instead determine complete gap sets in the middle-$\eps$ case and develop recursive outer descriptions of the full step-size set, including new explicit blackout intervals.  
The exact value four on $3-2\sqrt2<\eps\le1/3$ complements the endpoint value six in
\cite{TianLouShang}.  
In the positive-scale results, the underlying two-set construction, its qualitative positive-thickness conclusion, the gap-size estimate used in the thickness calculation, and the three-set mechanism are due to Hunt--Kan--Yorke \cite{hunt_kan_yorke}.  What is derived here is an explicit initial-stage lower bound within that construction, together with the numerical three-term, four-term, and asymmetric-pattern consequences.
We do not claim to introduce the general problem of step-size sets; the Fourier-analytic literature includes, in particular, \cite{CarnovaleSenger}.

Figure~\ref{fig:parameter-summary} records the main parameter ranges for
middle-$\eps$ Cantor sets.  Section~\ref{sec:prelim} reviews thickness and proves the qualitative four-point existence result used later.  Sections~\ref{sec:longest-AP} and~\ref{sec:blackout-intervals} contain the exclusion results.  
Section~\ref{sec:HKY section triple intersection} develops the quantitative Hunt--Kan--Yorke input and its applications: the three- and four-term interval theorems and the three-point asymmetric patterns.

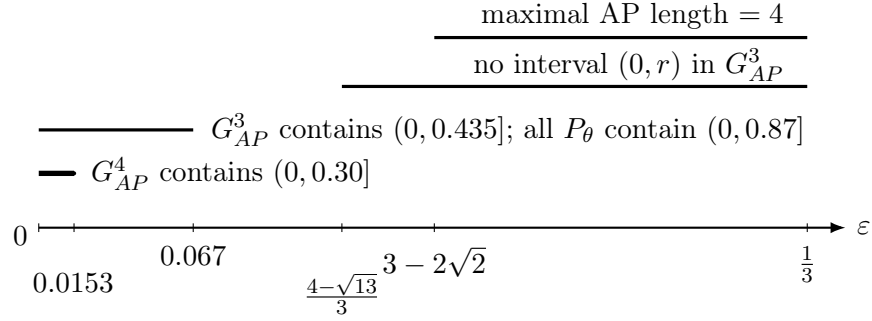
\begin{figure}[t]
\centering
\begin{tikzpicture}[x=30.5cm,y=0.72cm,>=Latex,font=\small]
  \def\xmax{0.3333333333}
  \draw[thick,-{Latex[length=2mm]}] (0,0)--(0.35,0) node[right] {$\eps$};
  \foreach \x in {0,0.0153,0.067,0.1314829081,0.1715728753,0.3333333333}
    {\draw (\x,0.08)--(\x,-0.08);}
  \node[below=3pt,anchor=east] at (0,0) {$0$};
  \node[below=13pt] at (0.015,0) {$0.0153$};
  \node[below=3pt] at (0.067,0) {$0.067$};
  \node[below=13pt] at (0.1314829081,0) {$\frac{4-\sqrt{13}}3$};
  \node[below=3pt] at (0.1715728753,0) {$3-2\sqrt2$};
  \node[below=3pt] at (0.3333333333,0) {$\frac13$};
  \draw[line width=2pt] (0,1.0)--(0.0153,1.0);
  \fill (0.0153,1.0) circle[radius=0.8pt];
  \node[anchor=west] at (0.018,1.0) {$G^4_{AP}$ contains $(0,0.30]$};
  \draw[line width=1.2pt] (0,1.8)--(0.067,1.8);
  \node[anchor=west] at (0.070,1.8) {$G^3_{AP}$ contains $(0,0.435]$; all $P_\theta$ contain $(0,0.87]$};
  \draw[line width=1.2pt] (0.1314829081,2.6)--(0.3333333333,2.6);
  \node[anchor=east] at (0.328,3.0) {no interval $(0,r)$ in $G^3_{AP}$};
  \draw[line width=1.2pt] (0.1715728753,3.5)--(0.3333333333,3.5);
  \node[anchor=east] at (0.328,3.9) {maximal AP length $=4$};
\end{tikzpicture}
\caption{Principal conclusions for the middle-$\eps$ family.  
A bar whose left endpoint is $0$ records a conclusion valid for every smaller positive
$\eps$.  
The endpoints shown in the positive-interval results are conservative
rounded constants; strict and non-strict parameter endpoints are as stated in the theorems.}
\label{fig:parameter-summary}
\end{figure}

\section{Preliminaries}\label{sec:prelim}

We review thickness, iterated function systems, and related results, as these notions will be used throughout the paper.
For a nonempty bounded set $E\subset\R$, write $|E|=\diam(E)$.  
A common way to formulate thickness is:

\begin{definition}(bridge thickness)
    Let $C\subset\R$ be a compact set, and let $(G_n)$ be the bounded connected components of the complement, $\R\setminus C$, ordered by nonincreasing length. 
    For $n\in\N$, each bounded gap $G_n$ is removed from a closed parent interval $I_n$, leaving behind two closed children intervals $L_n$ and $R_n$ called bridges, the left and right parts of $I_n\setminus G_n$, respectively. 
    The \textit{Newhouse thickness} of $C$ is defined by 
    \begin{equation*}
        \tau\lp C\rp:= \inf_{n\in\N} \frac{\min\left\{|L_n|,|R_n|\right\}}{|G_n|}.
    \end{equation*}
\end{definition}

As an alternative to bridges and gaps, Hunt--Kan--Yorke defined chunks; a set $P$ is a \emph{chunk} of $C$, written $P\propto C$, if $P=C\cap I$ is a nonempty proper set where $I$ is a closed interval and $d(P,C\setminus P)>0$.

\begin{definition}(chunk thickness)
    Given a compact set $C\subset \R$, we define the thickness of $C$ to be 
    \begin{equation*}
        \tau(C) = \inf_{P\propto C} \frac{|P|}{d(P,C\setminus P)}
    \end{equation*}
    provided $C$ has a chunk. 
    Otherwise, let $\tau(C) = 0$ if $C$ is empty or consists of a single point, and $\tau(C)=\infty$ if $C$ is an interval with positive length. 
\end{definition}

Note that these definitions of thickness are equivalent; see \cite{hunt_kan_yorke} for details.

We often want to know when two sets intersect; a precursor to which is being interleaved.
Two compact sets are \emph{interleaved} if each meets the interior of  the convex hull of the other.
Knowing two sets are interleaved and have sufficiently large thickness is enough to guarantee they intersect:

\begin{proposition}[Newhouse Gap Lemma \cite{newhouse}]
\label{prop:newhouse-gap-lemma}
    If interleaved compact sets $K,L\subset\R$ satisfy $\tau(K)\tau(L)\ge1$, then $K\cap L\ne\varnothing$.
\end{proposition}

\begin{lemma}[Short translates are interleaved]
\label{lem:short-translate-interleaved}
Let $K\subset\R$ be a nonempty compact set.  If
$0<u<\diam(K)$, then $K$ and $K+u$ are interleaved.
\end{lemma}

\begin{proof}
Write $\conv(K)=[a,b]$.  
Since $a,b\in K$ and $b-a>u$,
$$b\in K\cap(a+u,b+u)
\quad\text{and}\quad
a+u\in(K+u)\cap(a,b).$$
Thus each set meets the interior of the convex hull of the other.
\end{proof}

Consequently, if $C$ is compact, $d=\diam(C)>0$, and $\tau(C)\ge1$, then
Proposition~\ref{prop:newhouse-gap-lemma} applied to $C$ and $C+t$ gives
$d\subseteq G^2_{AP}(C)$.  
A related Gap Lemma argument shows that every compact set of thickness at least one contains every prescribed nondegenerate three-point configuration on the line; see
\cite[Proposition~20]{Yavicoli_Survey} for arithmetic progressions and
\cite{Sandberg_Taylor} for the general three-point statement.  
Quantitative existence results for longer finite patterns under stronger thickness hypotheses are developed in \cite{yavicoli_patterns}.  
These are qualitative or existence-oriented results; our focus is the geometry of the full scale set.

The symmetry of $C_\eps$ gives a useful four-point existence result.

\begin{proposition}[Symmetric four-point patterns]
\label{prop:symmetric-four-point-pattern}
Let $C\subset\R$ be a Cantor set with $\conv(C)=[0,1]$.  
Assume that $C$ is symmetric about $1/2$, that $1/2\notin C$, and that $\tau(C)\ge1$.  
If $g$ is the length of the largest bounded gap of $C$, then for every $u\in[g,1)$ the set $C$ contains an affine copy of
$$\{-1,-u,u,1\}.$$
Equivalently, for
$$Q_u=\left\{0,\frac{1-u}{2},\frac{1+u}{2},1\right\},$$
one has $G_{Q_u}(C)\ne\varnothing$.
\end{proposition}

\begin{proof}
Put
$$K=C-\frac12,\qquad L=\frac1u\left(C-\frac12\right).$$
Both sets have thickness $\tau(C)$, and
$$ \conv(K)=\left[-\frac12,\frac12\right] \subseteq\left[-\frac1{2u},\frac1{2u}\right]=\conv(L).$$
The set $L$ cannot lie in a bounded gap of $K$ because $|L|=1/u>1=|K|$.
If $u>g$, every bounded gap of $L$ has length at most $g/u<1=|K|$, so $K$ cannot lie in a bounded gap of $L$ either.  
Thus $K$ and $L$ are interleaved, and Proposition~\ref{prop:newhouse-gap-lemma} gives $K\cap L\ne\varnothing$.
For $z\in K\cap L$, the points $\pm z,\pm uz$ lie in $K$; since $0\notin K$, they form the required nondegenerate pattern after translating by $1/2$.

It remains to consider $u=g$.  
The same argument applies unless $K$ is contained in the closure of a bounded gap of $L$.  
In that exceptional case, the gap has length at most $g/u=1$, while $|K|=1$, so its closure must equal $[-1/2,1/2]$.
Scaling back shows that the central gap of $C$ has length $u$.
Its two endpoints, together with $0$ and $1$, form exactly the normalized
pattern $Q_u$.
\end{proof}

\begin{corollary}\label{cor:four-term-AP-middle-epsilon}
For $0<\eps<1$, the set $C_\eps$ contains a four-term arithmetic progression if and only if $\eps\le1/3$.
\end{corollary}

\begin{proof}
If $\eps\le1/3$, then $\tau(C_\eps)\ge1$ and the largest gap has length $\eps$.  Proposition~\ref{prop:symmetric-four-point-pattern}, with $u=1/3$, produces the normalized four-term pattern $\mathcal A_4$.

Conversely, suppose that $C_\eps$ contains a three-term progression.  Pull it back from the smallest construction interval containing it.  At the first level, two consecutive points lie in one interval of length $\lambda$, so the common difference is at most $\lambda$; another consecutive pair crosses the central gap, so it is at least $\eps$.  
Hence $\eps\le\lambda$, equivalently $\eps\le1/3$.  
A four-term progression contains a three-term one, which proves
the converse.
\end{proof}

These results are due to the strong self-similarity properties of the middle-$\eps$ Cantor sets which leads us to consider a more general family of self-similar compact sets constructed using iterated function systems.

\begin{definition}[IFS]\label{def:ifs}
    Fix $I\subset\R$.
    Define the \emph{iterated function system} $\{f_i\}_1^N$ to be a finite family of contractive maps $f_i:I\rightarrow I$ with contraction factor $0<r_i<1$; i.e., $|f_i(x)-f_i(y)|\leq r_i |x-y|$ for all $x,y\in I$, $1\leq i\leq N$.
    The \emph{attractor} of $\{f_i\}$ is the unique nonempty compact set $C\subset I$ satisfying 
    $$C=\bigcup_{i=1}^Nf_i(C).$$
    For a \emph{word} $w=(w_1,\ldots,w_n)\in\{0,1,\ldots,N\}^k$, write $f_w=f_{w_1}\circ\cdots\circ f_{w_k}$, and let $f_\emptyset$ be the identity.

    We say an IFS has \emph{strong separation} if the first-generation pieces $f_i(C)$ are pairwise disjoint. 
\end{definition}

\begin{remark}\label{rmk:ceps ifs}
    In the special case of the middle-$\eps$ Cantor set, the set $C_\eps$ is the attractor of the functions
    $$f_0(x)=\lambda x \quad\text{ and }\quad f_1(x)=1-\lambda+\lambda x,$$
    with domain $[0,1]$ and $\lambda = (1-\eps)/2$. 
    Its level-$n$ construction consists of $2^n$
    intervals of length $\lambda^n$, and
    $$\tau(C_\eps)=\frac{\lambda}{\eps}=\frac{1-\eps}{2\eps}.$$
\end{remark}

The above definition preserves the notion of parent and children intervals mentioned above; for example, $I$ is the parent interval to the children intervals $f_1(I),f_2(I),\ldots,f_N(I)$.

\section{Longest arithmetic progressions in middle-$\eps$ Cantor sets}
\label{sec:longest-AP}

In this section we prove that, for
$\eps\in(3-2\sqrt2,1/3]$, the longest arithmetic progressions in $C_\eps$ have length exactly four.  
The key point is a precise description of the possible midpoints of three-term progressions.

\subsection{The middle-third Cantor set}

\begin{proposition}\label{prop:gapC3}
Let $C_{1/3}$ be the middle-third Cantor set.  Then
$$ G^4_{AP}(C_{1/3})=G^3_{AP}(C_{1/3})=\left\{\frac1{3^n}:n\in\N\right\},$$
and $G^k_{AP}(C_{1/3})=\varnothing$ for every $k\ge5$.  
Moreover, $G^3_{AP}(C_\eps)=\varnothing$ whenever $\eps>1/3$.
\end{proposition}

\begin{proof}
For $C_{1/3}$ we have $\lambda=\eps=1/3$, and
$$\lambda^{n-1}\left\{0,\frac13,\frac23,1\right\}\subset C_{1/3}\qquad(n\in\N).$$
Thus $\{3^{-n}:n\in\N\}$ is contained in both
$G^3_{AP}(C_{1/3})$ and $G^4_{AP}(C_{1/3})$.

Conversely, let $a<b<c$ be a three-term arithmetic progression in $C_{1/3}$, and let $I$ be the smallest construction interval containing all
three points.  
Write $I_0,I_1$ for the two children of $I$.  
Their lengths and the length of the gap between them are all $|I|/3$.  
If $a,b\in I_0$ and $c\in I_1$, then
$$b-a\le \frac{|I|}{3}\le c-b.$$
Since the two distances are equal, equality holds throughout, and the common difference is $|I|/3$.  
The case $a\in I_0$ and $b,c\in I_1$ is symmetric.
Thus every common difference is $3^{-n}$ for some $n\in\N$.

The four endpoints of $I_0$ and $I_1$ form a four-term progression with common difference $|I|/3$.  
A five-term progression cannot occur: after rescaling its smallest containing construction interval to $[0,1]$, a first-level split would force its common difference to be at least $1/3$, whereas its total span forces it to be at most $1/4$, a contradiction.

Finally, if $\eps>1/3$, then the first-level gap has length $\eps$ and each first-level interval has length $\lambda<\eps$.  
In any first-level splitting three-term progression, two of the points lie in the same first-level interval, forcing the common difference to be at most $\lambda$, while a consecutive pair crossing the central gap forces it to be at least $\eps$.
This contradiction proves $G^3_{AP}(C_\eps)=\varnothing$.
\end{proof}

\subsection{Midpoints of three-term progressions}
 
Define the midpoint set
$$M_{AP}(C):=\{x\in C:\ \exists t>0\text{ with }x-t,x,x+t\in C\}.$$

\begin{theorem}[Midpoints of three-term progressions]\label{thm:midpoints}
Let $0<\eps\le1/3$.  Then
$$
M_{AP}(C_\eps)
=
\bigcup_{n\in\Nzero}
\ \bigcup_{w\in\{0,1\}^n}
 f_w\left(
 C_\eps\cap
 \left[\frac{1-\lambda}{2},\frac{1+\lambda}{2}\right]
 \right),
$$
where the $f_w$ are as in Remark \ref{rmk:ceps ifs}.
More precisely, the midpoints of the three-term progressions that split
between the two first-level intervals are exactly
$$ C_\eps\cap\left[\frac{1-\lambda}{2},\frac{1+\lambda}{2}\right].$$
\end{theorem}

\begin{proof}
Suppose first that $\{x-t,x,x+t\}\subset C_\eps$ splits at the first level.  
By symmetry, assume that $x\in[0,\lambda]$.  
Then
$x-t\in[0,\lambda]$ and $x+t\in[1-\lambda,1]$.  Hence
$$1-\lambda-x\le t\le x,$$
so $x\ge(1-\lambda)/2$.  
Therefore
$$x\in C_\eps\cap\left[\frac{1-\lambda}{2},\lambda\right].$$
Reflecting about $1/2$ gives the corresponding right-hand interval, proving necessity.

Conversely, let
$$x\in C_\eps\cap\left[\frac{1-\lambda}{2},\lambda\right].$$
If $x=(1-\lambda)/2$, take $t=x$.  
Then $x-t=0$ and $x+t=1-\lambda$, so the required progression is immediate.
We may therefore assume $x>(1-\lambda)/2$.  
When $\eps=1/3$, the interval under consideration contains only $x=1/3$, already covered by the boundary case; hence we may also assume $\eps<1/3$, or equivalently $\eps<\lambda$.

Define
$$A_x:=x-\bigl(C_\eps\cap[0,\lambda^2]\bigr),\qquad B_x:=\bigl(C_\eps\cap[1-\lambda,1]\bigr)-x.$$
These are affine copies of $C_\eps$, with
$$\tau(A_x)=\tau(B_x)=\frac{\lambda}{\eps}>1,$$
and
$$\conv(A_x)=[x-\lambda^2,x],
\qquad\conv(B_x)=[1-\lambda-x,1-x].$$
The strict inequality $x>(1-\lambda)/2$ implies that the two convex hulls overlap nontrivially. 
The largest bounded gap of $A_x$ has length $\lambda^2\eps<|B_x|=\lambda$, while the largest bounded gap of $B_x$ has length $\lambda\eps<|A_x|=\lambda^2$.  
Consequently neither set can lie in a bounded gap of the other; the hull overlap excludes the unbounded gaps.
Thus $A_x$ and $B_x$ are interleaved, and applying the Gap Lemma gives $A_x\cap B_x\ne\varnothing$.

For $t\in A_x\cap B_x$ we have
$$ x-t\in C_\eps\cap[0,\lambda^2], \qquad x+t\in C_\eps\cap[1-\lambda,1].$$
The two points lie in different first-generation intervals, so $t>0$.
Thus $x$ is the midpoint of a first-level splitting progression.  
The right-hand half follows by symmetry.

Finally, every three-term progression is contained in a smallest construction interval $f_w([0,1])$, determined by a word $w$. 
Pulling back by the corresponding inverse branch $f_w^{-1}$ produces a first-level splitting progression in $[0,1]$, and applying the branch $f_w$ again recovers the original interval $f_w[0,1]$.
This gives the stated union over words.
\end{proof}

\begin{lemma}[Endpoint localization]\label{lem:left-endpoint-localization}
Assume $3-2\sqrt2<\eps\le1/3$.  Suppose
$$\{x-t,x,x+t\}\subset C_\eps$$ splits at the first level and $$x\in C_\eps\cap\left[\frac{1-\lambda}{2},\lambda\right].$$
Then $x-t\in[0,\lambda^2]$.
\end{lemma}

\begin{proof}
The point $x$ belongs to the right child of the left first-level interval, since
$$\frac{1-\lambda}{2}\ge\lambda-\lambda^2.$$
If $x-t$ also belonged to that child, then
$$x-t\in[\lambda-\lambda^2,x),$$ so $t\le\lambda^2$.  
Since the progression splits at the first level,
$x+t\in[1-\lambda,1]$, and $x\le\lambda$ gives $t\ge1-2\lambda=\eps$.
Thus $\eps\le\lambda^2$, contrary to
$$\eps>3-2\sqrt2\quad\Longleftrightarrow\quad\eps>\lambda^2.$$
Hence $x-t$ lies in the other second-level child, namely $[0,\lambda^2]$.
\end{proof}

\begin{theorem}[Sharp four-term threshold]\label{thm:maxAP4}
Let $\eps\in(3-2\sqrt2,1/3]$.  Then the longest arithmetic progression
contained in $C_\eps$ has length exactly four.
\end{theorem}

\begin{proof}
Corollary~\ref{cor:four-term-AP-middle-epsilon} gives a four-term
progression, so it remains to exclude five-term progressions.

Suppose, to the contrary, that
$$a,a+t,a+2t,a+3t,a+4t\in C_\eps$$
for some $t>0$.  By passing to the smallest construction interval containing the progression and applying its inverse branch, we may assume that the progression splits at the first level.  
Set $x=a+2t$.  
The triple $$\{x-2t,x,x+2t\}$$ then splits at the first level.  
By symmetry and Theorem~\ref{thm:midpoints}, we may assume
$$x\in C_\eps\cap\left[\frac{1-\lambda}{2},\lambda\right].$$
Lemma~\ref{lem:left-endpoint-localization}, applied with common difference $2t$, gives $$x-2t\in[0,\lambda^2].$$
Also $x+2t$ lies in the right first-level interval.

There are two cases.  
If $x+t$ lies in the right first-level interval, then
$\{x-t,x,x+t\}$ splits at the first level, so Lemma \ref{lem:left-endpoint-localization} gives $x-t\in[0,\lambda^2]$.  
Since we already know $x-2t\in[0,\lambda^2]$, the two consecutive points $x-2t$ and $x-t$ lie in an interval of length $\lambda^2$; hence
$$ t=(x-t)-(x-2t)\le\lambda^2.$$
On the other hand, $x+t\ge1-\lambda$ and $x\le\lambda$, so $t\ge\eps$, contradicting $\eps>\lambda^2$.

If $x+t$ lies in the left first-level interval, then
$\{x,x+t,x+2t\}$ splits at the first level.  Its midpoint $x+t$ lies in the
left midpoint interval, and the localization lemma gives
$x\in[0,\lambda^2]$.  
This contradicts
$$x\ge\frac{1-\lambda}{2}>\lambda^2,$$
where the strict inequality holds for every $\lambda<1/2$.

Thus no five-term progression exists.  
By the nesting $G^{k+1}_{AP}(C)\subseteq G^k_{AP}(C)$, no longer progression exists either, and the maximal length is four.
\end{proof}

\section{Blackout intervals for pattern scale sets}
\label{sec:blackout-intervals}

This section considers when attractors of iterated function systems contain scale sets
$$G_P(C)=\{s>0:\exists x\in \R\text{ such that } x+sP\subset C\},$$
and arithmetic-progression gap sets
$$G_{AP}^k(C)= \{t>0:\exists x\in \R \text{ with } x, x+t, \ldots,x+(k-1)t\in C\},$$ as in equation \eqref{eq:def-pattern-scale} and \eqref{eq:def-gap-set}, respectively.

A copy $x+sP$ {first splits} at level $n+1$ when it lies completely in one level-$n$ interval but intersects at least two children of that interval. 

An open interval disjoint from $G_P(C)$ will be called a \emph{blackout interval}.  
The preceding notation permits a single first-splitting argument for arithmetic and asymmetric patterns.

\begin{theorem}[General pattern blackout bound]\label{thm:general-pattern-blackout}
Let $C$ be the attractor of an affine IFS with $M\ge2$,
$$ f_i(x)=r_i x+a_i,\qquad 1\le i\le M,\qquad 0<r_i<1,$$ with strong separation, $\conv(C)=[0,1]$, and
$$f_i([0,1])=[a_i,b_i],\qquad r_i=b_i-a_i,$$
where
$$0=a_1<b_1<a_2<b_2<\cdots<a_M<b_M=1.$$
Set
$$g_{\min}:=\min_{2\le j\le M}(a_j-b_{j-1}).$$
Let $P=\{0=p_1<\cdots<p_k=1\}$ be a normalized $k$-point pattern, and write $\delta_{\min}=\delta_{\min}(P)$ and $\delta_{\max}=\delta_{\max}(P)$.  Then
$$ G_P(C)\subseteq
\bigcup_{n\in\Nzero} \ \bigcup_{(i_1,\ldots,i_n)\in\{1,\ldots,M\}^n} (r_{i_1}\cdots r_{i_n})[L_P,U_P],$$
where
$$L_P:=\frac{g_{\min}}{\delta_{\max}}$$
and
$$U_P:=
\begin{cases}
\displaystyle
\min\left\{
1,
\frac{\sum_{i=1}^M r_i}{\delta_{\min}(k-M)},
\frac{\max_{1\le i\le M}r_i}
{\delta_{\min}(\lceil k/M\rceil-1)}
\right\},&k>M,\\[3ex]
1,&2\le k\le M.
\end{cases}
$$
\end{theorem}

\begin{proof}
Let $D_P$ be the set of scales $s\in G_P(C)$ for which some copy $x+sP\subset C$ is split at the first level.  
Every copy either is already split or lies in a single first-level interval.  In the latter case, applying
the inverse branch produces another copy of $P$ in $C$.  
This iteration must terminate after finitely many steps.  
Indeed, let $r_{\max}=\max_i r_i<1$.  
If the original copy of scale $s>0$ remained in a single construction interval for the first $n$ levels, then that interval would have diameter at most $r_{\max}^n$, and therefore $s\le r_{\max}^n$.
This is impossible for every $n$, since $r_{\max}^n\to0$.  
Thus a first split occurs after finitely many inverse branches, and
$$G_P(C)\subseteq
\bigcup_{n\in\Nzero}
\ \bigcup_{(i_1,\ldots,i_n)\in\{1,\ldots,M\}^n}
(r_{i_1}\cdots r_{i_n})D_P.$$

Take $s\in D_P$.  
At least one pair of consecutive points of the embedded pattern lies in different first-level intervals.  
For some$j\in\{2,\ldots,k\}$,
$$\delta_j(P)s\ge g_{\min},$$
which implies
$$s\ge\frac{g_{\min}}{\delta_j(P)}\ge\frac{g_{\min}}{\delta_{\max}}=L_P.$$
Since $x+sP\subset[0,1]$ and $\diam(P)=1$, we also have $s\le1$.

Assume $k>M$.  
If $A_i=(x+sP)\cap[a_i,b_i]$, then points of $A_i$ are separated by at least $\delta_{\min}s$, and hence
$$
\#A_i\le
\left\lfloor\frac{r_i}{\delta_{\min}s}\right\rfloor+1.
$$
Therefore
$$
k\le\sum_{i=1}^M\#A_i
\le\frac{\sum_{i=1}^M r_i}{\delta_{\min}s}+M,
$$
which yields
$$
s\le\frac{\sum_{i=1}^M r_i}{\delta_{\min}(k-M)}.
$$
Moreover, some first-level interval contains at least $\lceil k/M\rceil$ points.  
Because an interval contains a consecutive block of the ordered pattern, those points span at least $\lceil k/M\rceil-1$ consecutive pattern gaps.  
Thus
$$
(\lceil k/M\rceil-1)\delta_{\min}s
\le\max_{1\le i\le M}r_i,
$$
which gives the third upper bound.  
Hence $D_P\subseteq[L_P,U_P]$, and the first-splitting decomposition proves the theorem.
\end{proof}

\begin{corollary}[Homogeneous IFS]\label{cor:homogeneous-pattern-blackout}
Under the hypotheses of Theorem~\ref{thm:general-pattern-blackout}, suppose $r_i=r$ for every $i$.  
Then
$$G_P(C)\subseteq\bigcup_{n\in\Nzero}r^n[L_P,U_P].$$
If $L_P>U_P$, then $G_P(C)=\varnothing$.  
If $rU_P<L_P$, then the intervals $r^n[L_P,U_P]$ are pairwise disjoint and $G_P(C)$ contains no interval of
the form $(0,\eta)$.
\end{corollary}

\begin{proof}
Every word of length $n$ has contraction ratio $r^n$, so the asserted containment is the specialization of Theorem~\ref{thm:general-pattern-blackout}.  
If $L_P>U_P$, the first-split scale set is empty and hence so is $G_P(C)$.  
If $rU_P<L_P$, then for every $n\ge0$,
$$r^{n+1}U_P<r^nL_P,$$
so consecutive containing intervals are separated by the nonempty open gap
$(r^{n+1}U_P,r^nL_P)$.  
These gaps accumulate at zero, which rules out an interval $(0,\eta)$ in $G_P(C)$.
\end{proof}

\begin{corollary}[Arithmetic progressions]\label{cor:general-AP-blackout}
Under the IFS hypotheses of Theorem~\ref{thm:general-pattern-blackout}, for
$k\ge3$,
$$
G^k_{AP}(C)\subseteq
\bigcup_{n\in\Nzero}
\ \bigcup_{(i_1,\ldots,i_n)\in\{1,\ldots,M\}^n}
(r_{i_1}\cdots r_{i_n})[g_{\min},c_k],
$$
where
$$
c_k:=
\begin{cases}
\displaystyle
\min\left\{
\frac1{k-1},
\frac{\sum_{i=1}^M r_i}{k-M},
\frac{\max_{1\le i\le M}r_i}{\lceil k/M\rceil-1}
\right\},&k>M,\\[3ex]
\displaystyle\frac1{k-1},&2\le k\le M.
\end{cases}
$$
\end{corollary}

\begin{proof}
Apply Theorem~\ref{thm:general-pattern-blackout} to $P=\mathcal A_k$.  
Here $\delta_{\min}=\delta_{\max}=1/(k-1)$ and \eqref{eq:AP-pattern-conversion} gives $G_{\mathcal A_k}(C)=(k-1)G^k_{AP}(C)$.  
Dividing the resulting interval bounds by $k-1$ proves the claim.
\end{proof}

\begin{corollary}[Small-scale exclusion criterion]\label{cor:middle-epsilon-AP-blackout}
Let $0<\eps\le1/3$, $\lambda=(1-\eps)/2$, and $k\ge3$.  
Then
$$
G^k_{AP}(C_\eps)
\subseteq
\bigcup_{n\in\Nzero}\lambda^n[\eps,c_k(\eps)],
$$
where
$$
c_k(\eps):=
\min\left\{
\frac1{k-1},
\frac{2\lambda}{k-2},
\frac{\lambda}{\lceil k/2\rceil-1}
\right\}.
$$
If $\eps>c_k(\eps)$, then $G^k_{AP}(C_\eps)=\varnothing$.  
If $\lambda c_k(\eps)<\eps$, then $G^k_{AP}(C_\eps)$ contains no interval of the form $(0,r)$.  
In particular, for $k=3$ we have $c_3(\eps)=\lambda$, so this conclusion holds whenever
$$
\eps>\lambda^2
\quad\Longleftrightarrow\quad
\eps>3-2\sqrt2.
$$
\end{corollary}

\begin{proof}
The middle-$\eps$ IFS has $M=2$, common contraction ratio $\lambda$, and first-level gap $g_{\min}=\eps$.  Substituting these data into Corollary~\ref{cor:general-AP-blackout} gives the stated formula for $c_k(\eps)$ and the scale containment.  The two conclusions follow from Corollary~\ref{cor:homogeneous-pattern-blackout}.  
For $k=3$, $\lambda\le1/2$ gives
$$
c_3(\eps)=\min\left\{\frac12,2\lambda,\lambda\right\}=\lambda.
$$
Finally,
$$
\eps>\lambda^2
\quad\Longleftrightarrow\quad
4\eps>(1-\eps)^2
\quad\Longleftrightarrow\quad
\eps^2-6\eps+1<0,
$$
which, in $0<\eps\le1/3$, is equivalent to $\eps>3-2\sqrt2$.
\end{proof}

\begin{corollary}[Asymmetric three-point patterns]\label{cor:asymmetric-pattern-blackout}
Let $0<\eps\le1/3$, $\lambda=(1-\eps)/2$, and $P_\theta=\{0,\theta,1\}$ with $1/2\le\theta<1$.  
Then
$$
G_{P_\theta}(C_\eps)
\subseteq
\bigcup_{n\in\Nzero}\lambda^n
\left[
\frac{\eps}{\theta},
U_\theta(\eps)
\right],
\qquad
U_\theta(\eps):=
\min\left\{1,\frac{\lambda}{1-\theta}\right\}.
$$
If
$$
\lambda U_\theta(\eps)<\frac{\eps}{\theta},
$$
then $G_{P_\theta}(C_\eps)$ contains no interval of the form $(0,r)$.
The case $0<\theta<1/2$ follows by reflecting the pattern and replacing $\theta$ by $1-\theta$.
\end{corollary}

\begin{proof}
For $1/2\le\theta<1$,
$$
\delta_{\min}(P_\theta)=1-\theta,
\qquad
\delta_{\max}(P_\theta)=\theta.
$$
Theorem~\ref{thm:general-pattern-blackout}, with $M=2$, $r_1=r_2=\lambda$, and $g_{\min}=\eps$, therefore gives
$$
L_{P_\theta}=\frac{\eps}{\theta}
$$
and
$$
U_{P_\theta}
=\min\left\{1,
\frac{2\lambda}{1-\theta},
\frac{\lambda}{1-\theta}\right\}
=U_\theta(\eps).
$$
The blackout criterion is now
Corollary~\ref{cor:homogeneous-pattern-blackout}.  
If $0<\theta<1/2$, the
reflection $x\mapsto1-x$ carries $P_\theta$ onto $P_{1-\theta}$.  
Since $C_\eps$ is symmetric about $1/2$, reflection preserves its pattern scale sets, and the result follows from the case $1-\theta>1/2$.
\end{proof}

\subsection{A geometric refinement for three-term progressions}
\label{subsec:first-split-spectrum}

We now specialize to the middle-$\eps$ Cantor set from the local setup above.
The first-splitting argument records only the diameter and the central gap of the two first-generation intervals.  
For three-term progressions one can retain more of the product geometry which yields both an exact renormalization formula and a strictly sharper blackout criterion.

Write $C_0=f_0(C_\eps)$ and $C_1=f_1(C_\eps)$ for the two first-generation pieces.

\begin{definition}[First-split spectrum]\label{def:first-split-spectrum}
Define
$$
B_\eps:=\left\{t>0:
\begin{array}{l}
\text{there exists }x\text{ with }x,x+t,x+2t\in C_\eps,\\
\text{with at least one of these points in each }C_0, C_1
\end{array}
\right\}.
$$
Thus $B_\eps$ consists of the progression gaps whose first separation occurs at the first construction level.
\end{definition}

\begin{lemma}[Exact first-split decomposition]
\label{lem:exact-first-split-decomposition}
For every $0<\eps\le 1/3$,
$$
\boxed{
G^3_{AP}(C_\eps)=\bigcup_{n=0}^{\infty}\lambda^n B_\eps.
}
$$
\end{lemma}

\begin{proof}
This follows by self-similarity of the middle-$\eps$ Cantor set.
\end{proof}

We next compute a level-two outer approximation to $B_\eps$.  
Let $C_\eps^{(2)}$ be the union of the four level-two construction intervals,
which we denote by
$$
\begin{aligned}
I_{00}&=[0,\lambda^2],
& I_{01}&=[\lambda-\lambda^2,\lambda],\\
I_{10}&=[1-\lambda,1-\lambda+\lambda^2],
& I_{11}&=[1-\lambda^2,1],
\end{aligned}
$$
where
$$
I_0=[0,\lambda]\quad\text{ and }\quad I_1=[1-\lambda,1],
$$
are the two first-generation construction intervals.
Define
$$
B_\eps^{(2)}:=\left\{t>0:
\begin{array}{l}
\text{there exists }x\text{ with }x,x+t,x+2t\in C_\eps^{(2)},\\
\text{and these three points meet both }I_0\text{ and }I_1
\end{array}
\right\}.
$$
Since $C_\eps\subseteq C_\eps^{(2)}$ and
$C_i\subseteq I_i$ for $i=0,1$, we have
$$B_\eps\subseteq B_\eps^{(2)}.$$

For fixed $t$, let
$$
\ell_t^{(1)}=\{(u,v):v=u+t\},
\qquad
\ell_t^{(2)}=\{(u,v):v=u+2t\}.
$$
A progression $x,x+t,x+2t\in C_\eps^{(2)}$ exists exactly when the first-coordinate projections of
$$
\ell_t^{(1)}\cap(C_\eps^{(2)}\times C_\eps^{(2)})
\quad\text{and}\quad
\ell_t^{(2)}\cap(C_\eps^{(2)}\times C_\eps^{(2)})
$$
intersect.  
The following proposition computes the permitted parameters.

\begin{proposition}[Level-two outer approximation]
\label{prop:level-two-first-split}
Set
$$
\rho:=\frac{3-\sqrt5}{2},
\qquad
\eta:=\sqrt2-1,
\qquad
\kappa:=\frac{\sqrt{13}-1}{6},
$$
and
$$a_\lambda:=\frac{1-\lambda-\lambda^2}{2},
\qquad
b_\lambda:=\frac{1-\lambda+\lambda^2}{2},
\qquad
c_\lambda:=1-\lambda-\lambda^2.$$
For $1/3\le\lambda<1/2$,
$$
B_\eps^{(2)}=
\begin{cases}
[\eps,\lambda],
&\displaystyle \frac13\le\lambda\le\rho,\\[1ex]
[a_\lambda,b_\lambda],
&\rho<\lambda<\eta,\\[1ex]
[\eps,\lambda^2]\cup[a_\lambda,b_\lambda]\cup[c_\lambda,\lambda],
&\eta\le\lambda<\kappa,\\[1ex]
[\eps,\lambda],
&\kappa\le\lambda<\frac12.
\end{cases}
$$
At $\lambda=\eta$, the first and third intervals in the third line are singletons.
\end{proposition}

\begin{proof}
We first record a general interval calculation.  Let
$$
I=[p,p+L],\qquad J=[q,q+L],\qquad K=[r,r+L].
$$
The conditions $x\in I$, $x+t\in J$, and $x+2t\in K$ are equivalent to
$$
[p,p+L]\cap[q-t,q+L-t]\cap[r-2t,r+L-2t]\ne\varnothing.
$$
Comparing every lower endpoint with every upper endpoint shows that the admissible $t\ge0$ form
\begin{equation}\label{eq:three-interval-feasibility}
\begin{split}
\Big[&\max\left\{q-p-L,\frac{r-p-L}{2},r-q-L,0\right\},\\
&\min\left\{q-p+L,\frac{r-p+L}{2},r-q+L\right\}\Big].
\end{split}
\end{equation}
provided that the left endpoint does not exceed the right endpoint.

Because $C_\eps$ is symmetric about $1/2$, it is enough to analyze only the cases where $x\in I_{00}$ or $I_{01}$.
There are six nondecreasing level-two intervals.
Substituting their endpoints into \eqref{eq:three-interval-feasibility}, with $L=\lambda^2$, gives
$$\begin{array}{c|cc|c}
    (x,x+t,x+2t)&\text{candidate interval}&& \text{nonempty condition} \\ \hline\hline
   (I_{00},I_{00},I_{10})&[c_\lambda,\lambda^2]&& \lambda \geq 1/2 \\ \hline
   (I_{00},I_{00},I_{11})&[1-2\lambda^2,\lambda^2] &&\lambda \geq 1/\sqrt{3} \\ \hline
   (I_{00},I_{01},I_{10}) & [\eps,\lambda] &\text{if } \lambda\leq \rho\\ 
    &[a_\lambda,b_\lambda] &\text{if } \lambda \geq \rho  \\ \hline
    (I_{00},I_{01},I_{11})&[c_\lambda,\lambda]&&\lambda \geq \eta \\ \hline
    (I_{01},I_{01},I_{10})&[\eps,\lambda^2]&&\lambda \geq \eta \\ \hline
    (I_{01},I_{01},I_{11})&[c_\lambda,\lambda^2] &&\lambda \geq 1/2. 
\end{array}$$

As $1/3\leq \lambda<1/2$, only the middle three rows can contribute.  
For the central row, the switch occurs when
$$
a_\lambda=\eps\quad\text{and}\quad b_\lambda=\lambda,
$$
which is equivalent to $\lambda=\rho$.  The two side intervals first appear at $\lambda=\eta$, because
$$
\eps=\lambda^2\quad\text{and}\quad c_\lambda=\lambda
\quad\Longleftrightarrow\quad \lambda=\eta.
$$
Finally,
$$
\lambda^2=a_\lambda
\quad\Longleftrightarrow\quad
b_\lambda=c_\lambda
\quad\Longleftrightarrow\quad
3\lambda^2+\lambda-1=0,
$$
whose positive root is $\kappa$.  
At and beyond this value the three
intervals overlap and their union is $[\eps,\lambda]$.  
This proves the four cases in the statement.  
For ease of independent verification, Appendix~\ref{app:level-two-certificate} records all six endpoint comparisons before simplification.
\end{proof}

The level-two geometry improves the previous threshold.  
Define
$$
\eps_*:=\frac{4-\sqrt{13}}{3}=0.1314829081\ldots
<3-2\sqrt2=0.1715728753\ldots.
$$

\begin{theorem}[Improved blackout threshold]
\label{thm:improved-middle-epsilon-blackout}
If
$$
\eps_*<\eps\le\frac13,
$$
then $G^3_{AP}(C_\eps)$ contains no interval of the form $(0,r)$.
Moreover, in the additional parameter range
$$
\eps_*<\eps\le 3-2\sqrt2
\qquad\bigl(\eta\le\lambda<\kappa\bigr),
$$
one has the explicit blackout intervals
$$
\boxed{
\left(\lambda^{n+2},
\frac{1-\lambda-\lambda^2}{2}\,\lambda^n\right)
\cap G^3_{AP}(C_\eps)=\varnothing
\qquad(n\ge0).
}
$$
\end{theorem}

\begin{proof}
When $\eps>3-2\sqrt2$, the conclusion follows already from Corollary~\ref{cor:middle-epsilon-AP-blackout}.  
It remains to consider $\eps_*<\eps\le3-2\sqrt2$, equivalently $\eta\le\lambda<\kappa$.  Proposition~\ref{prop:level-two-first-split} gives
\begin{equation}\label{eq:level-two-containment}
B_\eps\subseteq
[\eps,\lambda^2]\cup[a_\lambda,b_\lambda]
\cup[c_\lambda,\lambda],
\qquad
\lambda^2<a_\lambda.
\end{equation}
We also need to rule out interference from larger scales.  
On $[\eta,\kappa]$,
\begin{equation}\label{eq:scale-separation-polynomial}
2(\eps-\lambda a_\lambda)
=\lambda^3+\lambda^2-5\lambda+2>0.
\end{equation}
Indeed, the polynomial on the right is decreasing on $[0,1/2]$, and at $\lambda=\kappa$ it equals $(82-22\sqrt{13})/27>0$.

Fix $n\ge0$ and let
$$
J_n=(\lambda^{n+2},a_\lambda\lambda^n).
$$
By Lemma~\ref{lem:exact-first-split-decomposition}, it is enough to show that $J_n$ misses $\lambda^mB_\eps^{(2)}$ for every $m\ge0$.
If $m\ge n+1$, then
$$
\max(\lambda^mB_\eps^{(2)})\le\lambda^{m+1}
\le\lambda^{n+2}.
$$
If $m=n$, the exclusion follows directly from \eqref{eq:level-two-containment}.  
Finally, if $m\le n-1$, then by \eqref{eq:scale-separation-polynomial},
$$
\min(\lambda^mB_\eps^{(2)})
=\eps\lambda^m
\ge\eps\lambda^{n-1}
>a_\lambda\lambda^n.
$$
Thus $J_n\cap G^3_{AP}(C_\eps)=\varnothing$ for every $n$.  
Since the intervals $J_n$ accumulate at zero, $G^3_{AP}(C_\eps)$ cannot contain any interval $(0,r)$.
\end{proof}

\begin{remark}
The threshold $\eps_*$ is the strongest one obtainable by combining the exact first-split decomposition with this complete level-two outer containment alone.  
Indeed, when $\lambda\ge\kappa$ (equivalently $\eps\le\eps_*$),
Proposition~\ref{prop:level-two-first-split} gives $B_\eps^{(2)}=[\eps,\lambda]$, exactly the original first-level bound.
A further improvement would therefore require level three or deeper
symbolic information.
\end{remark}

\section{The Hunt-Kan-Yorke Method and Intervals of Pattern Scales}\label{sec:HKY section triple intersection}

In this section, we build on a result of Hunt-Kan-Yorke \cite{hunt_kan_yorke} to investigate the gap set 
\begin{equation*}
    G^k_{AP}(C):= \left\{ t>0 :\, \exists \, x\, \text{ so that } x,x+t,\ldots,x+(k-1)t \in C \right\}
\end{equation*}
to quantify the range of common differences that can occur in $3$- and $4$-term arithmetic progressions contained in an admissible compact set $C\subset \R$; in particular, we provide explicit conditions for the gap sets $G^3_{AP}(C)$ and $G^4_{AP}(C)$ to contain intervals about $0$.

\begin{theorem}[Interior of the $3$-Gap Set]\label{thm:hky triple intersection}
    Let $C\subset \R$ be a Cantor set with $d=\diam(C)>0$ and largest bounded gap $G$ satisfying $|G|\leq 0.067d$ and thickness $$\tau(C)\geq \frac{1-0.067}{2(0.067)} =  6.96268\ldots$$ 
    Then $G_{AP}^3(C)$ contains the interval $(0,0.435d]$.
\end{theorem}

As we will see in the subsections that follow, the key mechanism used to prove Theorem \ref{thm:hky triple intersection} is applying Proposition \ref{hky corollary 6} to the sets $C_1= C-s$, $C_2= C+s$, and $C_3 = C$, where $s>0$, from which it follows $s\in G_{AP}^3(C)$.
Note, however, that Proposition \ref{hky corollary 6} is incredibly flexible in the choice of $C_i$, and one could also choose $C_1=C+s\theta$, $C_2= C-s(1-\theta)$, and $C_3=C$, where $0<\theta<1$. 
Thus it follows as a corollary that $C_{P_\theta}(C)$ also contains an interval.

\begin{corollary}[Interior of the Asymmetric $3$-Point Patterns]\label{cor:asymmetric patterns}
    Let $C\subset \R$ be a Cantor set with $d=\diam(C)>0$ and largest bounded gap $G$ satisfying $|G|\leq 0.067d$ and thickness
    $$\tau(C)\geq \frac{1-0.067}{2(0.067)} = 6.96268\ldots$$
    Then $G_{P_\theta}(C)$ contains the interval $(0,0.87 d]$.
\end{corollary}

\begin{theorem}[Interior of the $4$-Gap Set]\label{thm:hky quadruple intersection}
    Let $C\subset \R$ be a Cantor set with $d=\diam(C)>0$ 
    and thickness $$\tau(C)\geq \frac{1-0.015}{2(0.015)} =  32.83333\ldots $$ 
    Then $G_{AP}^4(C)$ contains the interval $(0,0.3d]$.
\end{theorem}

This process is split into two parts: \S\ref{subsec:HKY} which provides a self-contained argument building on the work of Hunt-Kan-Yorke \cite{hunt_kan_yorke} to provide explicit and computable conditions for three and four Cantor sets to intersect and \S\ref{subsec:nonempty gap set} which applies the results of the previous subsection to prove Theorems \ref{thm:hky triple intersection} and \ref{thm:hky quadruple intersection}.

\subsection{Intersection of Three Cantor Sets}\label{subsec:HKY}
In this section, we improve a result of \cite{hunt_kan_yorke}:

\begin{proposition}[Theorem 2, \cite{hunt_kan_yorke}]\label{hky theorem 2}
    There is a function $\varphi(\tau_1,\tau_2)$ which is positive in regions $III$ and $IV$ of Figure \ref{fig:HKY regions} such that for all interleaved compact sets $C_1,C_2\subset\R$ with $\tau(C_1)\geq \tau_1$ and $\tau(C_2)\geq \tau_2$, there is a set $K\subset C_1\cap C_2$ with thickness at least $\varphi(\tau_1,\tau_2)$. 
\end{proposition}

In practice, this proposition only establishes that $\tau(C)>0$.
To better explain why this is and how we will improve this proposition, we provide a brief overview of its proof.
We postpone introducing some notation until after.

\begin{proof}[Outline of Proof.] 
Let $C_1$ and $C_2$ be sets of sufficiently large thickness; i.e., $\tau_1:=\tau(C_1)$ and $\tau_2:=\tau(C_2)$ lie in regions $III$ or $IV$ of Figure \ref{fig:HKY regions}. 
    \begin{enumerate}[label=(\roman*.)]
        \item Construct a new gap set $(K_n)$ which contains the gaps of $C_1$ and $C_2$. 
        Consequently, $K:=\lp \cup_n K_n\rp^\complement$ is contained in $C_1\cap C_2$.
       \item
        Argue the thickness of $K$ can be determined using the gaps $(K_n)$; i.e.,
        \begin{equation}\label{eq:bound on C}
            \tau(K)\geq\frac{d(K_n,K_m)}{|K_m|}.
        \end{equation}
        In fact, taking the minimum over all such $n,m$ there will exist some function $\varphi(\tau_1,\tau_2)$ such that 
        \begin{equation*}
            \tau(K) \geq \varphi(\tau_1,\tau_2).
        \end{equation*}
        \item
        Bound the denominator of equation \eqref{eq:bound on C} using Lemma \ref{hky lemma 4}.
      \item
      Bound the numerator of equation \eqref{eq:bound on C} using (\cite{hunt_kan_yorke}, Lemma 5). 
      \item\label{enum:combine} Combine the above upper and lower bounds to get 
      \begin{equation*}
          \tau(K) \geq \varphi(\tau_1,\tau_2):= \frac{t_*\lp(\tau_1-t_*)(\tau_2-t_*)-(1+t_*)^2\rp }{(\tau_1-t_*)(\tau_2-t_*)+(1+t_*)^2\lp 2\max(\tau_1,\tau_2)+1\rp }\psi_{t_*}(\tau_1,\tau_2)
      \end{equation*}
      where $0<t_*<\frac{\tau_1\tau_2-1}{\tau_1+\tau_2+2}$ is chosen to maximize $\varphi(\tau_1,\tau_2)$ and 
      \begin{equation*}
          \psi_t(\tau_1,\tau_2):=(1-\lambda_1-\lambda_2)^N\min \left\{ 1-\sigma_1\lambda_1-\sigma_2b_*,1-\sigma_2\lambda_2-(1-\lambda_2)\frac{\sigma_1\lambda_1}{1+\lambda_1} \right\}
      \end{equation*}
      where for $i=1,2$
    \begin{align*}
        \lambda_i := \frac{1+t}{\tau_i+1+t}\quad\text{ and }\quad
        \sigma_i := \frac{(\tau_i-t)(\tau_{3-i}+1)}{(\tau_i-t)(\tau_{3-i}-t)-(1+t)^2}.
    \end{align*}
    \end{enumerate}  
\end{proof}

Let us elaborate on the function $\psi_t(\tau_1,\tau_2)$ from \ref{enum:combine}.
In regions $III$ or $IV$, (\cite{hunt_kan_yorke}, Lemma 5)
establishes the lower bound on the denominator from \eqref{eq:bound on C}; i.e.,
$$d(K_m,K_n)\geq (1-\lambda_1-\lambda_2)^N\min \left\{ 1-\sigma_1\lambda_1-\sigma_2b_*,1-\sigma_2\lambda_2-(1-\lambda_2)\frac{\sigma_1\lambda_1}{1+\lambda_1} \right\}=:\psi_t(\tau_1,\tau_2),$$
for some interdependent large $N\in \N$ and $b_*\in \R_+$.
Observe that as $N$ increases, $\psi_t(\tau_1,\tau_2)$ decreases exponentially to zero and, consequently, the function $\varphi(\tau_1,\tau_2)$---which provides a lower bound for the thickness of $C$---decreases to zero even faster. 
Thus, not only is $\psi_t(\tau_1,\tau_2)$ small, it is also incredibly difficult to calculate explicitly.
Even if one calculates $\psi_t(\tau_1,\tau_2)$ explicitly, they must then substitute it into $\varphi(\tau_1,\tau_2)$ and optimize for $t_*$. 
It is probable that $N=1$ is sufficient for thickness pairs in region $III$, though it would take further investigation to verify this claim.

To overcome these difficulties, we ask what constraints would allow us to consider the simpler $N=0$ case. 
This leads to the discovery of region $IV$ of Figure \ref{fig:HKY regions}, and we prove in Lemma \ref{lem:HKY lemma 5 simplified} that
\begin{equation*}
    \tilde{\psi}_t(\tau_1,\tau_2) := 1-\sigma_1\lambda_1-\sigma_2\lambda_2
\end{equation*}
is sufficient as a lower bound. 
We then substitute this new function into Proposition \ref{hky theorem 2} and drop the assumption that $\varphi(\tau_1,\tau_2)$ has been optimized by $t_*$; consequently, $\tilde{\varphi}_t(\tau_1,\tau_2)$ is calculable and there is a range of valid $t$ that makes $\tilde{\varphi}_t(\tau_1,\tau_2)$ positive and large.
\begin{proposition}\label{thm:explicit varphi revised}
    There is a function 
    \begin{equation}\label{defn:varphi}
        \tilde{\varphi}_t (\tau_1,\tau_2):= \frac{t\lp(\tau_1-t)(\tau_2-t)-(1+t)^2\rp }{(\tau_1-t)(\tau_2-t)+(1+t)^2\lp 2\max(\tau_1,\tau_2)+1\rp }\tilde{\psi}_{t}(\tau_1,\tau_2),
    \end{equation}
    where 
    \begin{equation*}
        \tilde{\psi}_{t}(\tau_1,\tau_2):= 1-\sigma_1\lambda_1-\sigma_2\lambda_2,
    \end{equation*}
    which is positive in region $IV$ for infinitely many $0<t< (\tau_1\tau_2-1)/(\tau_1+\tau_2+2)$ such that for all interleaved compact sets $C_1,C_2\subset\R$ with $\tau(C_1)\geq \tau_1$ and $\tau(C_2)\geq \tau_2$, there is a set $K\subset C_1\cap C_2$ with thickness at least $\tilde{\varphi}_t(\tau_1,\tau_2)$. 
\end{proposition}

\begin{remark}
    One can expand $\tilde{\psi}_t(\tau_1,\tau_2)$ to find
    \begin{align*}
        \tilde{\psi}_t(\tau_1,\tau_2) 
        &= 1-\frac{(\tau_1-t)(\tau_2+1)}{(\tau_1-t)(\tau_2-t)-(1+t)^2}\cdot \frac{1+t}{\tau_1+1+t}-\frac{(\tau_2-t)(\tau_1+1)}{(\tau_1-t)(\tau_2-t)-(1+t)^2}\cdot \frac{1+t}{\tau_2+1+t} \\
        &= \frac{ at^2 +bt +c }{\lb (\tau_1-t)(\tau_2-t)-(1+t)^2 \rb(\tau_1+1+t)(\tau_2+1+t)}
    \end{align*}
    where 
    \begin{align*}
        A &= -\lp 3\tau_1\tau_2+2\tau_1+2\tau_2+1\rp  \\
        B &= -\lp \tau_1^2\tau_2+\tau_1\tau_2^2+8\tau_1\tau_2+4\tau_1+4\tau_2+2\rp  \\
        C &= \tau_1^2\tau_2^2-4\tau_1\tau_2-2\tau_1-2\tau_2-1.
    \end{align*}
    Thus we let $t_{\text{max}}=\min\left\{ \frac{\tau_1\tau_2-1}{\tau_1+\tau_2+2}, \frac{-B+\sqrt{B^2-4AC}}{2A} \right\}$, and it follows that $\tilde{\psi}_t(\tau_1,\tau_2)$ is positive whenever $0<t<t_{\text{max}}$.
\end{remark}

\begin{remark}
    It is known that for $t>0$ sufficiently close to zero, $\tilde{\varphi}_t(\tau_1,\tau_2)$ is positive; however, the exact value of $t$ that maximizes $\tilde{\varphi}_t(\tau_1,\tau_2)$ increases as $\tau_1$, $\tau_2$ increase. 
    
    In our applications (\S\ref{subsec:nonempty gap set}), we fix $\tau_1=\tau_2$ and choose a particular $t$ to demonstrate large thickness $\tilde{\varphi}_t(\tau_1,\tau_1)$.
    We then extend this lower bound to all sets of thickness $\tau\geq\tau_1$. 
\end{remark}

Thus in what follows we construct the function $\tilde{\psi}_t(\tau_1,\tau_2)$, and the result that $$\tau(K)\geq \tilde{\varphi}_t(\tau_1,\tau_2)$$ follows for free.
Our outline is as follows.

In \S\ref{subsubsec:regions}, we introduce our revised thickness conditions for $(\tau_1,\tau_2)$; this is region $IV$ of Figure \ref{fig:HKY regions}.

In \S\ref{subsubsec:Kn}, we review how to construct the set $K$ in $C_1\cap C_2$. In \cite{hunt_kan_yorke}, this is accomplished by first taking the largest gap $G$ of $C_1$, $C_2$ and then expanding it to absorb nearby gaps that are within distance $t|G|$ for some arbitrary $0<t <\frac{\tau_1\tau_2-1}{\tau_1+\tau_2+2}$.
They then repeat this process until no gaps remain.
This creates a new system of gaps $(K_n)$, which are dependent on the number $t$; note that smaller values of $t$ make for smaller gaps $K_n$, while larger values of $t$ make for larger gaps $K_n$. 
Please note that we let the choice of $t$ be arbitrary; we can choose a particular value later.
$K$ is then the complement of the system of gaps $K_n$.

In \S\ref{subsubsec:hky lemma 4}, we restate Lemma \ref{hky lemma 4}. 
This lemma bounds the denominator and aids in bounding the numerator of \eqref{eq:bound on C} by establishing how far $K_n$ can extend past the gaps of $C_1$, $C_2$ that it contains.
As we use this lemma without modification, we leave it as a black box.

In \S\ref{subsubsec:lemma 5 modified}, we prove  our key Lemma \ref{lem:HKY lemma 5 simplified} which bounds the numerator of \eqref{eq:bound on C}; i.e.,
\begin{equation*}
    d(K_n,K_m) \geq \tilde{\psi}_t(\tau_1,\tau_2)\, d(G,H),
\end{equation*}
where 
\begin{align*}
    \tilde{\psi}_t(\tau_1,\tau_2) &:= 1-\sigma_1\lambda_1-\sigma_2\lambda_2
\end{align*}
and for $i=1,2$
\begin{align*}
    \lambda_i := \frac{1+t}{\tau_i+1+t}\quad\text{ and }\quad
    \sigma_i := \frac{(\tau_i-t)(\tau_{3-i}+1)}{(\tau_i-t)(\tau_{3-i}-t)-(1+t)^2}.
\end{align*}

All of the framework included is to make this paper self-contained; we encourage interested readers to review \cite{hunt_kan_yorke} for further details.

\begin{figure}[h!]
    \centering
    {\includegraphics[height=3in]{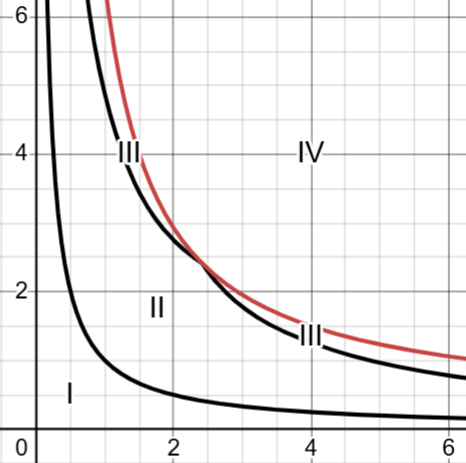}}
    
    \caption{The intersection of two interleaved compact sets with thicknesses $\tau_1$ and $\tau_2$: can be empty for $(\tau_1,\tau_2)\in I$; is nonempty for $(\tau_1,\tau_2)\in II$; contains a set of positive thickness  for $(\tau_1,\tau_2)\in III$; and contains a set with an explicit minimum thickness in $IV$.}
    \label{fig:HKY regions}
\end{figure}

\subsubsection{Regions}\label{subsubsec:regions}
In Figure \ref{fig:HKY regions}, the curve dividing regions $I$ and $II$ is given by $\tau_1\tau_2=1$; it is a consequence of the Newhouse Gap Lemma that sets $C_1$ and $C_2$ with thicknesses above the line $\tau_1\tau_2\geq 1$ have at least one point of intersection; additionally, this result is sharp.

The curves dividing regions $II$ and $III$ are given by the inequalities
\begin{equation*}
    \tau_1 \geq \tau_2, \quad \tau_1>\frac{\tau_2^2+3\tau_2+1}{\tau_2^2}, \quad  \text{ and }\quad  \tau_2 > \frac{(2\tau_1+1)^2}{\tau_1^3}, 
\end{equation*}
or
\begin{equation*}
    \tau_2 \geq \tau_1, \quad   \tau_2>\frac{\tau_1^2+3\tau_1+1}{\tau_1^2}, \quad \text{ and }\quad   \tau_1 > \frac{(2\tau_2+1)^2}{\tau_2^3}.
\end{equation*}
These inequalities come from the original argument in Hunt-Kan-Yorke \cite{hunt_kan_yorke}, and all thickness pairs $(\tau_1,\tau_2)$ in regions $III$ or $IV$ guarantee the intersection of $C_1$ and $C_2$ contain a set of positive thickness.
Moreover, this dividing line is sharp; see \cite{Williams}.

The red line dividing regions $III$ and $IV$ is our new contribution and the focus of this section, \S\ref{sec:HKY section triple intersection}; it is given by 
\begin{equation*}
    \tau_1^2\tau_2^2 - 4\tau_1\tau_2-2\tau_1-2\tau_2-1=0,
\end{equation*}
or equivalently 
\begin{equation*}
    \tau_i = \frac{1+2\tau_j+\sqrt{2\tau_j^3+5\tau_j^2+4\tau_j+1}}{\tau_j^2},
\end{equation*}
for $i\neq j$. 
\textbf{We show that all thickness pairs $(\tau_1,\tau_2)$ in region $IV$ not only contain a set of positive thickness, but we obtain an explicit lower bound for that thickness.}

\subsubsection{Construction of a new system of gaps $K_n$}\label{subsubsec:Kn}
We revisit how to construct the sets $(K_n)$.
Recall that we want the gaps $(K_n)$ to contain all the gaps of $C_1$ and $C_2$ from which it follows that $K\subset C_1\cap C_2$ where $K:=\lp\cup_n K_n\rp^\complement$.

Let $C_1$ and $C_2$ be interleaved compact sets with $\tau(C_1)\geq \tau_1$ and $\tau(C_2)\geq \tau_2$ for some $(\tau_1,\tau_2)$ in region $IV$ of Figure \ref{fig:HKY regions}.

Let the gaps of $C_1$ be $I_0,I_1,I_2,\ldots$, with $I_0$ and $I_1$ unbounded, $I_0$ to the left of $I_1$, and $|I_2|\geq |I_3|\geq\cdots$.
For $C_2$ we define $J_0,J_1,J_2,\ldots$ similarly.
We refer to the intervals $I_n$ and $J_n$ collectively as the ``original gaps.''
Our goal is to construct the complement of $K$ as a union of disjoint open intervals $K_0, K_1,K_2,\ldots$ with $K_0$ and $K_1$ unbounded, and with every original gap contained in some $K_m$ whence $K\subset C_1\cap C_2$. 
To get a lower bound on the thickness of $K$, 
observe that every chunk $P$ of $K$ is bordered on each side by a gap of $K$, with at least one of the bordering gaps being bounded. 
Pick a chunk $P$, and say $P$ is bordered by $K_m$ and $K_n$ with $m>n$ and $m\geq 2$. 
Then 
$$\frac{|P|}{d(P,K\setminus P)} = \frac{d(K_m,K_n)}{\min(|K_m|,|K_n|)} \geq \frac{d(K_m,K_n)}{|K_m|}.$$
Recalling that thickness is defined as 
$$\tau(K)=\inf_{P\propto K} \frac{|P|}{d(P,K\setminus P)}, $$
we see that finding a universal lower bound of $\frac{d(K_m,K_n)}{|K_m|}$ would ensure that the thickness is also bounded.
Hence, Hunt-Kan-Yorke show that there exists some $\tilde{\varphi}_t(\tau_1,\tau_2)>0$ such that whenever $m>n$ and $m\geq 2$,
\begin{equation}\label{eq: Km,Kn varphi bound}
    \frac{d(K_m,K_n)}{|K_m|} \geq \tilde{\varphi}_t(\tau_1,\tau_2). 
\end{equation}
In this paper, we focus on improving the bound of the numerator (i.e., bounding $d(K_m,K_n)$) through improving the bound on $\tilde{\psi}_t(\tau_1,\tau_2)$. 

We return to the construction of the $(K_n)$ by finding a pair of original gaps $I_*$ and $J_*$ between which $K$ will lie; that is $I_*$ and $J_*$ will be contained in $K_0$ and $K_1$. 
The properties we desire of $I_*$ and $J_*$ are that they are a positive distance apart, that all gaps of $C_1$ with an endpoint between the closures of $I_*$ and $J_*$ are bounded and no larger than $I_*$ and likewise (in comparison to $J_*$) for gaps of $C_2$ between $I_*$ and $J_*$. 
We will show later that once $I_*$ and $J_*$ have been determined, the diameter of $K$ can be bounded below by a constant depending on $\tau_1$ and $\tau_2$ times the distance between $I_*$ and $J_*$. 

\begin{figure}
\begin{tikzpicture}[
    line style/.style={dashed, thick},
    tick/.style={shift={(0,0.1)}, -={Bar[width=6pt]}}
]

    % ==========================================
    % CASE 1
    % ==========================================
    \node[scale=1.2] at (-1.5, 2.45) {Case 1 $\left\{\vphantom{\begin{array}{c}1\\2\end{array}}\right.$};
    \node at (-0.3, 2.8) {$S_1$:};
    \node at (-0.3, 2.1) {$S_2$:};

    % Case 1 - S1 intervals
    \draw[line style, <-] (0, 2.8) -- (4, 2.8);
    \draw[thick] (4, 2.9) -- (4, 2.7); % Manual precise vertical bar
    \node[below] at (2, 2.8) {$I_* = I_0$};
    
    \draw[line style, ->] (9, 2.8) -- (10.5, 2.8);
    \draw[thick] (9, 2.9) -- (9, 2.7); % Manual precise vertical bar
    \node[below] at (9.75, 2.8) {$I_1$};

    % Case 1 - S2 intervals
    \draw[line style, <-] (0, 2.1) -- (1.1, 2.1);
    \draw[thick] (1.1, 2.2) -- (1.1, 2.0);
    \node[below] at (0.55, 2.1) {$J_0$};
    
    \draw[line style, ->] (6.8, 2.1) -- (10.5, 2.1);
    \draw[thick] (6.8, 2.2) -- (6.8, 2.0);
    \node[below] at (8.15, 2.1) {$J_* = J_1$};

    % ==========================================
    % CASE 2 (Y-coordinates shifted up by 1.1 units)
    % ==========================================
    \node[scale=1.2] at (-1.5, 0.75) {Case 2 $\left\{\vphantom{\begin{array}{c}1\\2\end{array}}\right.$};
    \node at (-0.3, 1.1) {$S_1$:};
    \node at (-0.3, 0.4) {$S_2$:};

    % Case 2 - S1 intervals
    \draw[line style, <-] (0, 1.1) -- (4, 1.1);
    \draw[thick] (4, 1.2) -- (4, 1.0);
    \node[below] at (2, 1.1) {$I_0$};
    
    \draw[line style, ->] (6.8, 1.1) -- (10.5, 1.1);
    \draw[thick] (6.8, 1.2) -- (6.8, 1.0);
    \node[below] at (8.65, 1.1) {$I_* = I_1$};

    % Case 2 - S2 intervals
    \draw[line style, <-] (0, 0.4) -- (1.1, 0.4);
    \draw[thick] (1.1, 0.5) -- (1.1, 0.3);
    \node[below] at (0.55, 0.4) {$J_0$};
    
    \draw[line style] (3.7, 0.4) -- (5.7, 0.4);
    \draw[thick] (3.7, 0.5) -- (3.7, 0.3); % Left bound of J*
    \draw[thick] (5.7, 0.5) -- (5.7, 0.3); % Right bound of J*
    \node[below] at (4.7, 0.4) {$J_*$};
    
    \draw[line style, ->] (9, 0.4) -- (10.5, 0.4);
    \draw[thick] (9, 0.5) -- (9, 0.3);
    \node[below] at (9.75, 0.4) {$J_1$};

\end{tikzpicture}
\caption{Cases in the construction of $I_*$ and $J_*$.}\label{fig:hky6}
\end{figure}
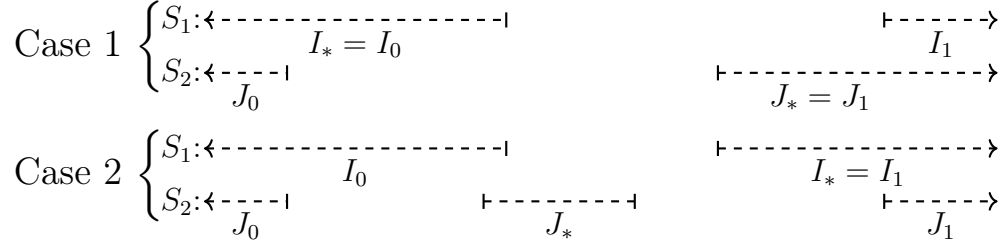

Assume without loss of generality that $J_0\subset I_0$. 
If $I_1\subset J_1$, Case 1 of Figure \ref{fig:hky6}, then $I_*=I_0$ and $J_*=J_1$ have the above properties; they must be separated by a positive distance since $C_1$ and $C_2$ are interleaved. 
If $J_1\subset I_1$, Case 2 of Figure \ref{fig:hky6}, let $J_*$ be the largest gap of $C_2$ with an endpoint between $I_0$ and $I_1$, and let $I_*$ be whichever of $I_0$ and $I_1$ is farthest from $J_*$. 
At least one of $I_0$ and $I_1$ must be a positive distance from $J_*$ since $C_1$ and $C_2$ are interleaved.

Next, let $t$ be a positive constant whose precise value will be chosen later; for now we assume that $0<t<(\tau_1\tau_2-1)/(\tau_1+\tau_2+2)<\min(\tau_1,\tau_2)$. 
Assume without loss of generality that $I_*$ lies to the left of $J_*$. 
We begin constructing $K_0$ by requiring that it contain $I_*$. 
We then require that $K_0$ contain the rightmost bounded $J_n$ with $d(I_*,J_n)\leq t|J_n|$ (we leave the verification that there is a rightmost gap satisfying this condition to the original authors). 
If there does not exist such a $J_n$ that is not already contained in $I_*$, we stop the construction and let $K_0=I_*$. 
Otherwise, we further require that $K_0$ contain the rightmost bounded $I_m$ that is within $t$ times its length of the previously added $J_n$. 
Again, if this requirement does not extend $K_0$ any farther rightward, we stop the construction.
If not, we then add to $K_0$ the rightmost $J_l$ which is within $t$ times its length of $I_m$ and is at most as large as $J_n$ (see Figure \ref{fig:hky7}). 
If a next step is necessary, we consider gaps of $C_1$ which are no larger than $I_m$, and so forth. 
We may have to continue this process infinitely often, but if so we must converge to a right endpoint for $K_0$, since there is no way this construction can extend past the rightmost point in $C_1\cup C_2$. 

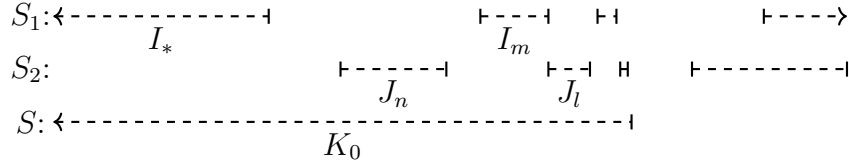
\begin{figure}
\centering
\begin{tikzpicture}[
    % Universal style for the interval lines
    line style/.style={dashed, thick}
]

    % Labels on the left
    \node at (-0.3, 1.4) {$S_1$:};
    \node at (-0.3, 0.7) {$S_2$:};
    \node at (-0.3, 0.0) {$S$:};

    % ==========================================
    % ROW S1
    % ==========================================
    % Interval I_*
    \draw[line style, <-] (0, 1.4) -- (2.85, 1.4);
    \draw[thick] (2.85, 1.5) -- (2.85, 1.3);
    \node[below] at (1.425, 1.4) {$I_*$};
    
    % Interval I_m
    \draw[line style] (5.65, 1.4) -- (6.55, 1.4);
    \draw[thick] (5.65, 1.5) -- (5.65, 1.3);
    \draw[thick] (6.55, 1.5) -- (6.55, 1.3);
    \node[below] at (6.1, 1.4) {$I_m$};
    
    % Small empty interval/ticks right after I_m
    \draw[line style] (7.2, 1.4) -- (7.45, 1.4);
    \draw[thick] (7.2, 1.5) -- (7.2, 1.3);
    \draw[thick] (7.45, 1.5) -- (7.45, 1.3);

    % Infinite ray at the end of S1
    \draw[line style, ->] (9.4, 1.4) -- (10.5, 1.4);
    \draw[thick] (9.4, 1.5) -- (9.4, 1.3);

    % ==========================================
    % ROW S2
    % ==========================================
    % Interval J_n
    \draw[line style] (3.8, 0.7) -- (5.2, 0.7);
    \draw[thick] (3.8, 0.8) -- (3.8, 0.6);
    \draw[thick] (5.2, 0.8) -- (5.2, 0.6);
    \node[below] at (4.5, 0.7) {$J_n$};
    
    % Interval J_l
    \draw[line style] (6.55, 0.7) -- (7.1, 0.7);
    \draw[thick] (6.55, 0.8) -- (6.55, 0.6);
    \draw[thick] (7.1, 0.8) -- (7.1, 0.6);
    \node[below] at (6.825, 0.7) {$J_l$};
    
    % Very thin empty interval right after J_l
    \draw[line style] (7.5, 0.7) -- (7.6, 0.7);
    \draw[thick] (7.5, 0.8) -- (7.5, 0.6);
    \draw[thick] (7.6, 0.8) -- (7.6, 0.6);

    % Bounded interval at the end of S2
    \draw[line style] (8.45, 0.7) -- (10.5, 0.7);
    \draw[thick] (8.45, 0.8) -- (8.45, 0.6);
    \draw[thick] (10.5, 0.8) -- (10.5, 0.6);

    % ==========================================
    % ROW S
    % ==========================================
    % Interval K_0
    \draw[line style, <-] (0, 0.0) -- (7.65, 0.0);
    \draw[thick] (7.65, 0.1) -- (7.65, -0.1);
    \node[below] at (3.825, 0.0) {$K_0$};

\end{tikzpicture}
\caption{The construction of $K_0$.}
\label{fig:hky7}
\end{figure}

We define $K_1$ similarly, starting with the requirement that $K_1$ contain $J_*$ and extending $K_1$ to the left if necessary in the same way we constructed $K_0$. 

Next, to construct $K_2$ we first require that it contain the largest original gap (choose any one in the case of a tie) not contained in $K_0\cup K_1$ (if no such gap exists, we leave $K_2$ undefined and let $S$ be the complement of $K_0\cup K_1$). 
Then we extend it on both the left and right in the same manner as before, but considering only gaps that are at most as large as the one we started with, to obtain the endpoints of $K_2$. 
We next start with the largest original gap not contained in $K_0\cup K_1\cup K_2$, proceeding similarly to define $K_3$, and so forth. 
Any given original gap must eventually be contained in some $K_n$ because there can only be finitely many original gaps that are as large or larger than the given one. 
We do not yet know that the $K_n$ are disjoint from each other; this will follow when we prove \eqref{eq: Km,Kn varphi bound}, though.

Let us now examine our construction more closely. 
Define $l(I)$ and $r(I)$ to be the left and right endpoints of an interval $I$, respectively. 
For a given $K_n$, let $G_0$ be the gap we started with in its construction, which for $n\geq 2$ must be the largest original gap it contains, or at least tied for the largest. 
For simplicity and without loss of generality we assume that $G_0$ is a gap of $C_1$. 
Consider the collection $E$ of all $J_n$ gaps of $C_2$ with $|J_n|\leq |G_0|$, $r(J_n)>r(G_0)$, and $d(G_0,J_n)\leq t|J_n|$. 
We claim---see \cite{hunt_kan_yorke} for details---that the members of $E$ are increasing in size from left to right.

We likewise define $G_2$ to be the rightmost gap of $C_1$ which is at most as large as $G_0$ and lies within $t$ times its length of $G_1$; again, if no such gap exists with $r(G_2)>r(G_1)$ we say that $|G_2|=0$ and $r(G_2)=r(G_1)$. 
Next, to define $G_3$ we consider only gaps of $S_2$ which are at most as large as $G_1$, for $G_4$ we look only at gaps of $C_1$ no larger than $G_2$, and so forth. 
Define $G_{-1}, G_{-2},\cdots$ similarly to be the leftmost (and largest) gaps added to $K_n$ at each stage of the process extending $K_n$ leftward. 
Then we may think of the open interval $K_n$ as being defined by 
\begin{equation*}
    l(K_n)=\lim_{m\rightarrow-\infty} l(G_m), \quad r(K_n)=\lim_{m\rightarrow\infty} r(G_m).
\end{equation*}
Each limit exists because it is the limit of a bounded monotonic sequence.

In the above construction, the even-numbered $G_m$ are gaps of $C_1$ and the odd-numbered ones are gaps of $C_2$, but if $G_0$ had been a gap of $C_2$ it would be the other way around.
In any case, $G_0$ is the largest even-numbered $G_m$ and either $G_1$ or $G_{-1}$ is the largest odd-numbered one. 
Also, the even-numbered $G_m$ decrease monotonically in size as one moves either rightward or leftward from the largest, and the same statement holds for the odd-numbered $G_m$. 
We call a given $G_m$ either a ``1-gap'' or ``2-gap'' of $K_n$ according to whether it is a gap of $C_1$ or $C_2$. 
Notice that not all original gaps contained in $K_n$ are 1-gaps or 2-gaps, only those that have been given a label $G_m$ in the construction of $K_n$. 
When we refer henceforth to left-to-right ordering or adjacency among the 1-gaps and 2-gaps of a given $K_n$, it is with respect to the ordering $\ldots, G_{-2},G_{-1},G_0,G_1,G_2,\ldots$. 
Thus, 1-gaps can only be adjacent to 2-gaps and vice versa.

This concludes our review of the construction of the gaps $(K_n)$. 
The above notation will be used in the following lemmas, and further information can be found in \cite{hunt_kan_yorke}.

\subsubsection{Black Box Lemma}\label{subsubsec:hky lemma 4}
The following lemma is used in bounding both the numerator and denominator of the left side of inequality \eqref{eq: Km,Kn varphi bound}.
For all $m\geq 0$, it establishes a bound on how far $K_n$ can extend to the right of $G_m$ in terms of how far $G_{m+1}$ extends past $G_m$, and similarly for $m\leq 0$ on the left. 
\begin{lemma}[Lemma 4, \cite{hunt_kan_yorke}]\label{hky lemma 4}
    Assume $0<t<(\tau_1\tau_2-1)/(\tau_1+\tau_2+2) $. 
    Let $$\sigma_1=\frac{(\tau_1-t)(\tau_2+1)}{(\tau_1-t)(\tau_2-t)-(1+t)^2}$$ 
    and
    $$\sigma_2 = \frac{(\tau_2-t)(\tau_1+1)}{(\tau_1-t)(\tau_2-t)-(1+t)^2}.$$
    Let $G$ be a $1$-gap of $K_n$ which is at least as large as all $1$-gaps of $K_n$ to its right. 
    Let $H$ be the next $2$-gap of $K_n$ to the right of $G$. 
    Then $$r(K_n)-r(G)\leq \sigma_2(r(H)-r(G)).$$
    The same statement with ``$1$'' and ``$2$'' interchanged holds, as do the corresponding results for left endpoints.
\end{lemma}

Lemma \ref{hky lemma 4} is a key tool in proving both the original Lemma 5, \cite{hunt_kan_yorke} and our modification (Lemma \ref{lem:HKY lemma 5 simplified}) below; hence, we will treat this lemma as a black box in what follows. 

\subsubsection{Explicit formula for $\tilde{\psi}_t(\tau_1,\tau_2)$}\label{subsubsec:lemma 5 modified}
As established above, we want to show that the left side of \eqref{eq: Km,Kn varphi bound} is positive for all $m>n$ and $m\geq 2$. 
Lemma \ref{lem:HKY lemma 5 simplified} provides an improved bound for the denominator $d(K_n,K_m)$ in region $IV$:

\begin{lemma}[]\label{lem:HKY lemma 5 simplified}
    There exists a function 
    \begin{align*}
        \tilde{\psi}_t(\tau_1,\tau_2) &:= 1-\sigma_1\lambda_1-\sigma_2\lambda_2
    \end{align*}
    that is positive whenever $(\tau_1,\tau_2)\in IV$, for which the following statement holds.
    For $m\neq n$, let $G$ be a $1$-gap of $K_m$ and $H$ be a $2$-gap of $K_n$. 
    If all $1$-gaps of $K_m$ or $K_n$ with at least one endpoint between the closures of $G$ and $H$ are bounded and at most as large as $G$, and all similarly situated $2$-gaps are bounded and at most as large as $H$, then 
    $$d(K_n,K_m)\geq \tilde{\psi}_t(\tau_1,\tau_2)d(G,H).$$
\end{lemma}

\begin{proof}[Proof of Lemma \ref{lem:HKY lemma 5 simplified}]
Let $G$ be a $1$-gap of $K_m$ and $H$ be a $2$-gap of $K_n$ satisfying the hypotheses. 
We assume without loss of generality that $(\tau_1,\tau_2)\in IV$, so 
$$\tau_i > \frac{1+2\tau_j+\sqrt{2\tau_j^3+5\tau_j^2+4\tau_j+1}}{\tau_j^2},$$
for $i\neq j$. 
If $d(G,H)=0$, the inequality to be proven is trivial. 
Otherwise, let us normalize $d(G,H)$ to one, and assume $G$ lies to the left of $H$. 
Let $G_0=G$ and $H_0=H$. 
Let $G_1$ be the $1$-gap of $K_n$ adjacent to $H_0$ on its left, and let $H_1$ be the $2$-gap of $K_m$ adjacent to $G_0$ on its right. 
Let $G_2$ be the adjacent $1$-gap of $K_m$ rightward from $H_1$, and likewise define $H_2, G_3, H_3,\ldots$.
For $i\geq 0$ let 
$$
x_i = \begin{cases}
l(H_i)-l(G_{i+1}), & i \text{ even}, \\
r(G_{i+1})-r(H_i), & i \text{ odd},
\end{cases}
$$
and 
$$
y_i = \begin{cases}
r(H_{i+1})-r(G_{i}), & i \text{ even}, \\
l(G_{i})-l(H_{i+1}), & i \text{ odd}.
\end{cases}
$$
Let $R_i=d(G_i,H_i)$; then $R_0=1$ and $R_{i+1}=\max(R_i-x_i-y_i,0)$ for $I\geq 0$. 
Let $R_\infty$ be the limit as $i$ goes to infinity of $R_i$. 
Then $d(K_m,K_n)=R_\infty$, so we wish to show that there is a positive lower bound on $R_\infty$ which depends only on $\tau$ and $t$. 

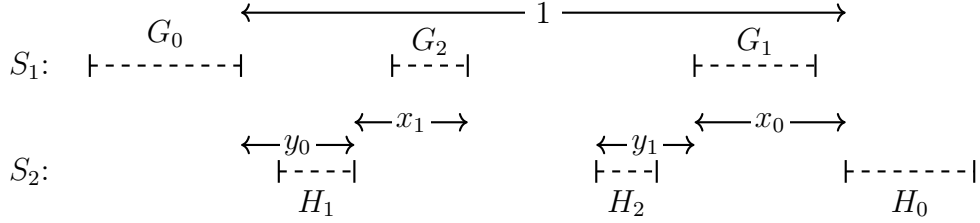
\begin{figure}
\begin{tikzpicture}[
    % Keep the style consistent with previous versions
    line style/.style={dashed, thick}
]

    % Labels on the left
    \node at (-0.3, 1.4) {$S_1$:};
    \node at (-0.3, 0.0) {$S_2$:};

    % ==========================================
    % TOP LAYOUT: Overall Dimension & S1 Intervals
    % ==========================================
    % Top overall dimension "1"
    \draw[<->, thick] (2.5, 2.1) -- (10.5, 2.1);
    \node[fill=white] at (6.5, 2.1) {1};

    % Interval G_0
    \draw[line style] (0.5, 1.4) -- (2.5, 1.4);
    \draw[thick] (0.5, 1.55) -- (0.5, 1.25);
    \draw[thick] (2.5, 1.55) -- (2.5, 1.25);
    \node[above] at (1.5, 1.5) {$G_0$};

    % Interval G_2
    \draw[line style] (4.5, 1.4) -- (5.5, 1.4);
    \draw[thick] (4.5, 1.55) -- (4.5, 1.25);
    \draw[thick] (5.5, 1.55) -- (5.5, 1.25);
    \node[above] at (5, 1.4) {$G_2$};

    % Interval G_1
    \draw[line style] (8.5, 1.4) -- (10.1, 1.4);
    \draw[thick] (8.5, 1.55) -- (8.5, 1.25);
    \draw[thick] (10.1, 1.55) -- (10.1, 1.25);
    \node[above] at (9.3, 1.4) {$G_1$};

    % ==========================================
    % MIDDLE LAYOUT: Clean X and Y Dimensions
    % ==========================================
    % LEVEL 1 (y = 0.65) -> x_1 and x_0 Level
    \draw[<->, thick] (4, 0.65) -- (5.5, 0.65);
    \node[fill=white, inner sep=1.5pt] at (4.75, 0.65) {$x_1$};

    \draw[<->, thick] (8.5, 0.65) -- (10.5, 0.65);
    \node[fill=white, inner sep=1.5pt] at (9.5, 0.65) {$x_0$};

    % LEVEL 2 (y = 0.35) -> y_0 and y_1 Level (Perfect Horizontal Shared Axis)
    \draw[<->, thick] (2.5, 0.35) -- (4, 0.35);
    \node[fill=white, inner sep=1.5pt] at (3.25, 0.35) {$y_0$};

    % Fixed: Spans exactly from left of H2 (7.2) to left of G1 (8.5)
    \draw[<->, thick] (7.2, 0.35) -- (8.5, 0.35);
    \node[fill=white, inner sep=1.5pt] at (7.85, 0.35) {$y_1$};

    % ==========================================
    % BOTTOM LAYOUT: S2 Intervals
    % ==========================================
    % Interval H_1
    \draw[line style] (3, 0.0) -- (4, 0.0);
    \draw[thick] (3, 0.15) -- (3, -0.15);
    \draw[thick] (4, 0.15) -- (4, -0.15);
    \node[below] at (3.5, -0.1) {$H_1$};

    % Interval H_2
    \draw[line style] (7.2, 0.0) -- (8, 0.0);
    \draw[thick] (7.2, 0.15) -- (7.2, -0.15);
    \draw[thick] (8, 0.15) -- (8, -0.15);
    \node[below] at (7.6, -0.1) {$H_2$};

    % Interval H_0
    \draw[line style] (10.5, 0.0) -- (12.2, 0.0);
    \draw[thick] (10.5, 0.15) -- (10.5, -0.15);
    \draw[thick] (12.2, 0.15) -- (12.2, -0.15);
    \node[below] at (11.35, -0.1) {$H_0$};

\end{tikzpicture}
\caption{The gaps $G_i$ and $H_i$.}\label{fig:hky9}
\end{figure}

We first find bounds for $x_i$ and $y_i$ in terms of each other. 
For all even $i\geq 0$, $G_i$ are $1$-gaps of $K_n$, $H_{i+1}$ are $2$-gaps of $K_m$, and $|G_i|\geq |G_{i+2}|$; see Figure \ref{fig:hky9}.
Then
\begin{equation*}
    \tau_1 |G_{i+2}| \leq d(G_i,G_{i+2}) \leq d(H_{i+1}, G_{i+2}) + r(H_{i+1})-r(G_i) \leq t |G_{i+2}| + r(H_{i+1})-r(G_i),
\end{equation*}
where the first inequality follows from the definition of thickness and the last inequality follows from construction of the $1$- and $2$-gaps; i.e., $G_{i+2}$ is the largest gap of $S_1$ with size at most $|G_i|$ within distance $t|G_{i+2}|$ of $H_{i+1}$.
Because $t\leq \tau_1$, we re-arrange as 
\begin{equation*}
    |G_{i+2}|\leq \frac{r(H_{i+1})-r(G_i)}{\tau_1-t}.
\end{equation*}
Hence, 
\begin{align*}
    x_{i+1}
    &=r(G_{i+2})-r(H_{i+1}) 
    \leq |G_{i+2}| + d(H_{i+1}, G_{i+2}) \\
    &\leq (1+t) |G_{i+2}| 
    \leq \frac{1+t}{\tau_1-t}\lp r(H_{i+1}) - r(G_i)\rp
    = \frac{1+t}{\tau_1-t}y_i. 
\end{align*}
A similar argument works for odd $i$; consequently,
\begin{equation}\label{eq:x i+1}
    x_{i+1}\leq\frac{1+t}{\tau_1-t}y_i.
\end{equation}
Similarly, for $y_i$
\begin{equation}\label{eq:y i+1}
    y_{i+1} \leq\frac{1+t}{\tau_2-t}x_i.
\end{equation}
Furthermore, by Lemma 4 \cite{hunt_kan_yorke} we have  
\begin{equation*}
    y_i+x_{i+1}+y_{i+2}+\cdots\leq \sigma_1 y_i
\end{equation*}
and
\begin{equation*}
    x_i+y_{i+1}+x_{i+2}+\cdots\leq \sigma_2 x_i.
\end{equation*}
Thus, for each $i$, 
\begin{equation}\label{eq:R infinity}
    d(K_m,K_n)=R_\infty \geq R_i-x_i-y_i-x_{i+1}-y_{i+1}-\cdots\geq R_i -\sigma_1x_i-\sigma_2y_i,
\end{equation}
We will show that for some $i$, the right side of \eqref{eq:R infinity} is positive. 
In fact, in region $IV$ we can take $i=0$. 

Let us obtain upper bounds on $x_0$ and $y_0$. 
We know 
\begin{equation}\label{eq:x0 - first}
    x_0=l(H_0)-l(G_1)\leq |g_1|+d(G_1,H_0)\leq (1+t)|G_1|,
\end{equation}
and by hypothesis $|G_1|\leq|G_0|$, so
\begin{equation}\label{eq:x0 - second}
    x_0=l(H_0)-r(G_0) -[l(G_1)-r(G_0)] = 1-d(G_0,G_1)\leq 1-\tau_1 |G_1|.
\end{equation}
Eliminating $|G_1|$ from these inequalities yields
\begin{equation}\label{eq:x0 upper bound}
    x_0\leq \frac{1+t}{\tau_1+1+t}=:\lambda_1.
\end{equation}
Similarly,
\begin{equation}\label{eq:y0 upper bound}
    y_0\leq \frac{1+t}{\tau_2+1+t}=:\lambda_2.
\end{equation}
We can obtain similar bounds on $x_i$ and $y_i$ for $i\geq 1$, but it is no longer necessary. 

We will show that at $t=0$, 
$$1-\sigma_1\lambda_1-\sigma_2\lambda_2>0,$$
from which it holds for $t$ sufficiently near $0$ by continuity. 
This in turn will imply 
$$d(K_m,K_n)\geq R_0-\sigma_1x_0-\sigma_2y_0 \geq1-\sigma_1\lambda_1-\sigma_2\lambda_2>0.$$
Observe that at $t=0$, 
\begin{align*}
    1-\sigma_1\lambda_1-\sigma_2\lambda_2 
    &= 1-\frac{\tau_1(\tau_2+1)}{\tau_1\tau_2-1}\frac{1}{\tau_1+1} - \frac{\tau_2(\tau_1+1)}{\tau_1\tau_2-1} \frac{1}{\tau_2+1} \\
    &= \frac{(\tau_1\tau_2-1)(\tau_1+1)(\tau_2+1)-\tau_1(\tau_2+1)^2-\tau_2(\tau_1+1)^2}{(\tau_1\tau_2-1)(\tau_1+1)(\tau_2+1)} \\
    &= \frac{\tau_1^2\tau_2^2-4\tau_1\tau_2-2\tau_1-2\tau_2-1}{(\tau_1\tau_2-1)(\tau_1+1)(\tau_2+1)}.
\end{align*}
Hence, $1-\sigma_1\lambda_1-\sigma_2\lambda_2>0$ if and only if $(\tau_i^2)\tau_j^2+(-2-4\tau_i)\tau_j + (-1-2\tau_i)>0$ for $i\neq j$ which is true so long as 
$$\tau_i> \frac{1+2\tau_j+\sqrt{2\tau_j^3+5\tau_j^2+4\tau_j+1}}{\tau_j^2},$$
for $i\neq j$. 
Thus, $1-\sigma_1\lambda_1-\sigma_2\lambda_2>0$ for $t=0$ and $(\tau_1,\tau_2)$ in region $IV$. 
Thus, $d(K_m,K_n)\geq 1-\sigma_1\lambda_1-\sigma_2\lambda_2>0$ for $t$ sufficiently near $0$. Recalling that $d(G,H)$ was normalized to $1$, the conclusion follows.  
\end{proof}

\begin{remark}\label{rmk:particular case of Lemma 5}
    We emphasize a particular case of Lemma \ref{lem:HKY lemma 5 simplified} as we use it in \S\ref{subsec:nonempty gap set}.

    Let $C_1$, $C_2$ be two compact sets with thicknesses in region $IV$ such that the intersection of their convex hulls, written $Q:=\conv(C_1)\cap\conv(C_2)$, contains neither $C_1$ nor $C_2$.
    Then Proposition \ref{thm:explicit varphi revised} applies and we obtain a compact set $K\subset C_1\cap C_2$ constructed from the gaps $(K_n)$; in particular, we can apply Lemma \ref{lem:HKY lemma 5 simplified} to the unbounded gaps $K_0$ and $K_1$ of $K$ from whence $$d(K_0,K_1)\geq \tilde{\psi}_t(\tau_1,\tau_2)d(I_*, J_*).$$ 
    Because $Q$ contains neither $C_1$ nor $C_2$, we are in Case 1 of Figure \ref{fig:hky6}; this implies $$|\conv(K)| \geq \tilde{\psi}_t(\tau_1,\tau_2)|Q|.$$
\end{remark}

\begin{remark}\label{rmk: tau1=tau2}
    In the special case where $\tau:=\tau_1=\tau_2$, we have
    \begin{equation}\label{eq:psit(tau)}
        \tilde{\psi}_t(\tau) 
        = 1-2\frac{\tau-t}{\tau-2t-1}\frac{1+t}{\tau+1+t} 
        = \frac{\tau^2-2\tau-1-t(3\tau+1)}{(\tau-2t-1)(\tau+1+t)},
    \end{equation}
    which is positive so long as $0< t < \frac{\tau^2-2\tau-1}{3\tau+1}$. 
    Consequently, referring to the definition of $\tilde{\varphi}_t(\tau_1,\tau_2)$ in \eqref{defn:varphi}, we can simplify $\tilde{\varphi}_t$ as 
\begin{equation}\label{eq:phit(tau)}
        \tilde{\varphi}_t(\tau)
        = \frac{t\lb (\tau-t)^2-(1+t)^2\rb \tilde{\psi}_{t}(\tau)}{(\tau-t)^2+(1+t)^2(2\tau+1)}  
         = \frac{t(\tau-2t-1)\tilde{\psi}_{t}(\tau)}{\tau+1+2t(t+1)},\\
         \end{equation}
    which is positive so long as $0<t<\frac{\tau-1}{2}$ and $\tilde{\psi}_t(\tau) >0$. 

   Putting this together, $\tilde{\varphi}_t(\tau)$ is positive so long as 
    $$0<t<\frac{\tau^2-2\tau-1}{3\tau+1}<\frac{\tau-1}{2}.$$
\end{remark}

As a corollary, Hunt-Kan-Yorke explain when the Gap Lemma can be applied to the set $K\subset C_1\cap C_2$ and some other set $C_3$. 
The argument follows by explaining how to apply the Gap Lemma \ref{prop:newhouse-gap-lemma} to the sets $K$ and $C_3$. 

\begin{proposition}[Corollary 6, \cite{hunt_kan_yorke}]\label{hky corollary 6}
    Let $C_1$ and $C_2$ be two interleaved compact sets whose thicknesses $(\tau_1,\tau_2)$ lie in region $IV$ and for which the intersection $Q$ of their convex hulls contains neither $C_1$ nor $C_2$. 
    If $C_3$ is a compact set with largest bounded gap $G$ such that
    \begin{enumerate}[label=(\roman*)]
        \item the convex hull of $C_3$ contains $Q$,
        \item $|G|<\tilde{\psi}_t(\tau_1,\tau_2)|Q|$, and
        \item $\tau(C_3)\tilde{\varphi}_t(\tau_1,\tau_2)\geq1$,
    \end{enumerate}
    for $0<t<t_{\text{max}}$.
    Then $C_1\cap C_2\cap C_3$ is nonempty.
\end{proposition}

\begin{proof}
    Indeed, let $C_1$, $C_2$ be two compact sets with thicknesses in region $IV$. 
    Apply Proposition \ref{thm:explicit varphi revised} to $C_1$ and $C_2$ to obtain the set $K\subset C_1\cap C_2$ with minimum thickness $\tilde{\varphi}_{t_*}(\tau_1,\tau_2)$, for some $0<t_*<t_{\text{max}}$. 
    To apply the Gap Lemma \ref{prop:newhouse-gap-lemma}, we need to establish that neither $C_3$ nor $K$ lie in a gap of the other and that the product of their thicknesses is greater than or equal to $1$. 
    
    To see that neither set lies in a gap of the other, we first need information about the diameter of $K$. 
    Recall that $K$ lies between the unbounded gaps $K_0$ and $K_1$, so we use Remark \ref{rmk:particular case of Lemma 5} to obtain 
    $$|\conv(K)| \geq \tilde{\psi}_t(\tau_1,\tau_2)|Q|,$$
    where $Q:=\conv(C_1)\cap \conv(C_2)$. 
    Thus condition (i) guarantees  
    $$C\subset Q\subset \conv(C_3),$$
    and condition (ii) guarantees the largest gap $G$ of $C_3$ is not wider than $\conv(K)$:
    $$|G|<\tilde{\psi}_t(\tau_1,\tau_2)|Q|<|\conv(K)|.$$
    Hence, we conclude that neither $K$ nor $C_3$ lie in a gap of the other. 
    
    Lastly, condition (iii) establishes that the product of the thicknesses is greater than one; i.e., 
    \begin{equation}\label{lookATthis}
    \tilde{\varphi}_t(\tau_1,\tau_2)\tau(K)\tau(C_3)\geq \tau(C_3)\geq 1.
    \end{equation}
    
    Thus we conclude $C_1\cap C_2\cap C_3$ is nonempty.
\end{proof}

An analogous argument can be used to obtain quadruple intersections $C_1\cap C_2\cap C_3\cap C_4\neq\varnothing$, for compact subsets $C_i\subset \R$, by applying Proposition \ref{hky corollary 6} to $C_1$, $C_2$ and $C_3$, $C_4$ to obtain new compact sets $K_1\subset C_1\cap C_2$ and $K_2\subset C_3\cap C_4$ of thickness $\tau(K_1),\tau(K_2)\geq 1$.
One can then apply the Gap Lemma \ref{prop:newhouse-gap-lemma} to $K_1$ and $K_2$ to show $C_1\cap C_2\cap C_3\cap C_4\neq\varnothing$.
However, this case requires strict conditions on the convex hulls $\conv(C_i)$ to ensure that the sets $K_1$ and $K_2$ are interleaved.
Thus, we show only the special case which proves the existence of $4$-term arithmetic progressions; i.e., $C_i:= C+ i s $ for $0\leq i\leq 3$. 

\begin{proposition}[four-term arithmetic progression criterion]
\label{prop:four-AP-HKY-general}
Let $C\subset\R$ be compact and $0<t<\frac{\tau^2-2\tau-1}{3\tau+1}$ be such that
$$\tilde{\varphi}_t(\tau)\geq1,$$
where $\tau:=\tau(C)$.
It necessarily follows that $\tau\gg 1+\sqrt{2}$.
Then for $d=\diam(C)>0$,
$$
\left[0,
\frac{\tilde{\psi}_t(\tau)}{1+2\tilde{\psi}_t(\tau)}d
\right]
\subset G^4_{AP}(C).
$$
\end{proposition}

\begin{proof}
We will show $C\cap (C+s)\cap (C+2s)\cap (C+3s)\neq \varnothing$ for $s\in \lb0,\frac{\tilde{\psi}_t(\tau)}{1+2\tilde{\psi}_t(\tau)}d\rb$ by first applying Proposition \ref{thm:explicit varphi revised} to the intersection $C\cap (C+2s)$. 
As thickness is invariant under similarities and arithmetic-progression gaps scale linearly, it is enough to treat the normalization $\conv(C)=[0,1]$.  
Because $0<t<\frac{\tau^2-2\tau-1}{3\tau+1}$, we know $\tilde{\psi}_t(\tau)>0$, and by hypothesis $\tilde{\varphi}_t(\tau)>1$.
Let
\begin{equation}\label{eq: s bounds}
0\leq s\leq \frac{\tilde{\psi}_t(\tau)}{1+2\tilde{\psi}_t(\tau)}.
\end{equation}
Since $\tilde{\psi}_t(\tau)/(1+2\tilde{\psi}_t(\tau))<1/2$, we have $s<1/2$.  
Therefore the sets
$$ C_0:=C,\qquad C_2:=C+2s $$
are interleaved: the right endpoint $1\in C_0$ lies in $(2s,1+2s)=\operatorname{int}\conv(C_2)$, and the left endpoint $2s\in C_2$ lies in $(0,1)=\operatorname{int}\conv(C_0)$.
Thus we apply Proposition \ref{thm:explicit varphi revised} to $C_0$ and $C_2$ and obtain a compact set
$$
K\subseteq C\cap(C+2s)
$$
for which
$$
\tau(K)\ge\tilde{\varphi}_t(\tau),
$$
and utilizing Remark \ref{rmk:particular case of Lemma 5}
$$
\diam(K)=|\conv(K)|\ge\tilde{\psi}_t(\tau)(1-2s).
$$
The choice of $s$ is in \eqref{eq: s bounds} equivalent to
$$
\tilde{\psi}_t(\tau)(1-2s)\geq s;
$$
consequently, $s\leq\diam(K)$.

We now investigate the set $(C+s)\cap (C+3s)$ by translating our set $K$ from above:
Translating the inclusion $K\subseteq C\cap(C+2s)$ gives
$$
K+s\subseteq(C+s)\cap(C+3s),
\qquad
\tau(K+s)=\tau(K)\ge\tilde{\varphi}_t(\tau).
$$
Because $s\leq\diam(K)$, Lemma~\ref{lem:short-translate-interleaved} shows that $K$ and $K+s$ are interleaved.  
Moreover,
$$
\tau(K)\tau(K+s)=\tau(K)^2\ge\tilde{\varphi}_t(\tau)^2\geq1.
$$
Then applying the Newhouse Gap Lemma gives a point
$$x\in K\cap(K+s).$$
Since $x\in K\subseteq C\cap(C+2s)$, we have
$$x\in C,\qquad x-2s\in C.$$
Since $x\in K+s$, the point $x-s$ belongs to $K$; consequently
$$x-s\in C,\qquad x-3s\in C.$$
Thus
$$x-3s,\quad x-2s,\quad x-s,\quad x$$
is a four-term arithmetic progression in $C$ with common difference $s$; hence $s\in G^4_{AP}(C)$.

Rescaling restores the factor $d$.
\end{proof}

This process of combining the Gap Lemma with the HKY criterion can be repeated for intersections of five or more sets; however, there is significant attrition of necessary information (e.g., exact length and position of convex hulls) that make this style of argument deteriorate. 
Thus we conclude our modifications of \cite{hunt_kan_yorke} here.

\subsection{Non-Empty Gap Set}\label{subsec:nonempty gap set}

In this section, we prove Theorem \ref{thm:hky triple intersection}; i.e., the gap set,
\begin{equation*}
    G^3_{AP}(C):= \left\{ t>0 :\, \exists \, x \text{ so that } x,x+t,x+2t \in C \right\},
\end{equation*}
for sufficiently thick $C$---with conditions on the largest gap $G_1$---contains the interval $(0,0.432]$.
With our new Lemma \ref{lem:HKY lemma 5 simplified} and Proposition \ref{thm:explicit varphi revised}, we now do so by utilizing Proposition \ref{hky corollary 6}:

Notice that if $(C-s)\cap C\cap (C+s)\neq \varnothing$ for a shift $s>0$, then $C$ contains a $3$-term arithmetic progression $\{c-s,c,c+s\}$ for some $c\in C$; hence, we consider the case where $C_1=C-s$, $C_2=C+s$, and $C_3=C$ for a shift $s>0$. 
Because thickness is translation invariant, we know $\tau:=\tau_1=\tau_2=\tau(C)$; hence, we use the simplified $\tilde{\varphi}_t(\tau)$ and $\tilde{\psi}_t(\tau)$ in equations \eqref{eq:phit(tau)} and \eqref{eq:psit(tau)}, respectively.

We prove Theorem \ref{thm:hky triple intersection} for a fixed thickness $\tau_\star$ and a fixed $t_\star$. 
Then, we extend to sets of larger thickness $\tau\geq\tau_\star$ by claiming that
$$\tilde{\psi}_{t_\star}(\tau)\geq \tilde{\psi}_{t_\star}(\tau_\star)\quad\text{and}\quad \tilde{\varphi}_{t_\star}(\tau)\geq \tilde{\varphi}_{t_\star}(\tau_\star,)$$
for all $\tau\geq \tau_\star$.

We begin by verifying the above claims. 

\begin{claim}\label{clm: psi increases for fixed t}
    For fixed $0<t<\frac{\tau-1}{2}$, $\tilde{\psi}_t(\tau)$ increases as $\tau$ increases.
\end{claim}

\begin{proof}
    Fix $0<t<\frac{\tau-1}{2}$. 
    Then by equation \eqref{eq:psit(tau)}, $$\tilde{\psi}_t(\tau) = \frac{\tau^2-(2+3t)\tau-1-t}{\tau^2-t\tau-(2t+1)(t+1)}. $$
    Taking the derivative with respect to $\tau$ yields
    \begin{align*}
        \tilde{\psi}_t'(\tau) &= \frac{\splitfrac{\lb\tau^2-t\tau-(2t+1)(t+1)\rb \lb 2\tau -(2+3t)\rb}{-\lb \tau^2 -(2+3t)\tau-1-t\rb\lb2\tau-t\rb}}{\lb \tau^2-t\tau-(2t+1)(t+1)\rb^2} \\
        &= \frac{2(1+t)\tau^2-t(4t+5)\tau +2(1+t)(3t^2+3t+1)}{\lb \tau^2-t\tau-(2t+1)(t+1)\rb^2}.
    \end{align*}
    Thus, $\tilde{\psi}_t'(\tau)$ is positive whenever $n_t(\tau) =2(1+t)\tau^2-t(4t+5)\tau +2(1+t)(3t^2+3t+1)$ is positive. 
    Observe that $n_t(\tau)$ is an upward facing parabola with negative discriminant $D= -32t^4-104t^3-135t^2-90t-16$; hence, $n_t(\tau)$ is always positive which means $\tilde{\psi}_t(\tau)$ is increasing for fixed $t$. 
\end{proof}

\begin{claim}\label{clm: varphi increases for fixed t}
    For fixed $0<t<\frac{\tau-1}{2}$, $\tilde{\varphi}_t(\tau)$ increases as $\tau$ increases. 
\end{claim}
\begin{proof}
    Fix $0<t<\frac{\tau-1}{2}$. 
    Recall from equation \eqref{eq:phit(tau)} that
    \begin{align*}
        \tilde{\varphi}_t(\tau) &= \frac{t(\tau^2-2\tau-1)-t^2(3\tau+1)}{(\tau+1)^2+3(\tau+1)t+2(\tau+2)t^2+2t^3} \\
        &= \frac{t\tau^2+(-3t^2-2t)\tau+(-t^2-t)}{\tau^2+(2t^2+3t+2)\tau +(2t^3+4t^2+3t+1)}.
    \end{align*}
    Taking the derivative with respect to $\tau$,
    \begin{align*}
        \tilde{\varphi}_t'(\tau) &= 
        \frac{\splitfrac{\lb \tau^2+(2t^2+3t+2)\tau + (2t^3+4t^2+3t+1)\rb \lb 2t\tau +(-3t^2-2t)\rb }{-\lb t\tau^2+(-3t^2-2t)\tau+(-t^2-t)\rb\lb 2\tau +(2t^2+3t+2)\rb }}{\lb \tau^2+(2t^2+3t+2)\tau +(2t^3+4t^2+3t+1)\rb^2} \\
        &= \frac{(2t^3+6t^2+4t)\tau^2 +(4t^4+8t^3+8t^2+4t)\tau+(-6t^5-14t^4-12t^3-4t^2)}{\lb\tau^2+(2t^2+3t+2)\tau +(2t^3+4t^2+3t+1)\rb^2}.
    \end{align*}
    Hence, $\tilde{\varphi}_t'(\tau)$ is positive when the numerator $$n_t(\tau) = (2t^3+6t^2+4t)\tau^2 +(4t^4+8t^3+8t^2+4t)\tau+(-6t^5-14t^4-12t^3-4t^2)$$ is positive.
    Note that $n_t(\tau)$ is an upward-facing parabola and $n_t(0)=-6t^5-14t^4-12t^3-4t^2<0$ for $0<t$. 
    Thus, if $n_t(\rho)>0$ for some $\rho>0$, then $n_t(\tau)>0$ for all $\tau \geq \rho$. 
    Observe that $0<t<\frac{\tau-1}{2}$ is equivalent to $1<2t+1<\tau$. 
    Then 
    \begin{align*}
        n_t(\tau) &= (2t^3+6t^2+4t)\tau^2 +(4t^4+8t^3+8t^2+4t)\tau+(-6t^5-14t^4-12t^3-4t^2) \\
        &> (2t^3+6t^2+4t)(2t+1)^2 +(4t^4+8t^3+8t^2+4t)(2t+1)+(-6t^5-14t^4-12t^3-4t^2) \\
        &= 10t^5+38t^4+54t^3+34t^2+8t 
        >0.
    \end{align*}
    Thus for fixed $t$ the function $\tilde{\varphi}_t(\tau)$ increases as $\tau$ increases.
\end{proof}

We now have the prerequisite tools to prove that the gap set of sufficiently thick $C$ with sufficiently small largest gap contains an interval including $0$.

\begin{proof}[Proof of Theorem \ref{thm:hky triple intersection}]
    Let $C$ be a compact set of thickness 
    $$\tau:=\tau(C)\geq 6.96,$$
    where, by similarity invariance of thickness, we normalize $\conv(C)=[0,1]$.
    Let the largest gap $G$ satisfy $|G| \leq 0.067$ and take $s\in (0,0.435]$.
    We will apply Proposition \ref{hky corollary 6} to the sets 
    $$C_1 = C -s, \quad C_2= C+s,\quad  \text{and}\quad C_3 = C$$ using the function $\tilde{\psi}_t(\tau)$ from Lemma \ref{lem:HKY lemma 5 simplified} to obtain $C_1\cap C_2\cap C_3\neq \varnothing$; in particular, that $G_\text{AP}^3(C) \supset (0,0.435]$.
    We do so by comparing against the anchor values $\tau_\star = 6.96$, $|G_\star|=0.067$, and $t_\star=0.55$ and using Claims \ref{clm: psi increases for fixed t} and \ref{clm: varphi increases for fixed t} to extend to sets of larger thickness.

    First, note that $C_1$ and $C_2$ are interleaved because $s<1/2$ and because $Q=[s,1-s]$, we know $Q$ contains neither $C_1$ nor $C_2$.

    (i) Observe that $\conv(C_3) = [0,1]\supset [s,1-s] = Q$. 

    (ii) Next, we want $|G|<\tilde{\psi}_{t_\star}(\tau) |Q|$ which is equivalent to 
    $$\tilde{\psi}_{t_\star}(\tau) > \frac{|G|}{1-2s},$$
    which holds for our chosen $t_\star$, $|G_\star|$, and $s\in(0,0.435]$ as 
    \begin{equation}\label{eq:triple int cor6 (ii)}
        \tilde{\psi}_{t_\star}(\tau) \geq \tilde{\psi}_{t_\star}(\tau_\star) = 0.51954\ldots \geq 0.51538\ldots = \frac{|G_\star|}{1-2s} \geq \frac{|G|}{1-2s}.
    \end{equation}

    (iii) Because $\tau(C_3)= \tau \geq 6.96$, the condition that $\tau(C_3)\tilde{\varphi}_{t_\star}(\tau(C_1),\tau(C_2))\geq1$ is equivalent to 
    \begin{equation*}
        \tilde{\varphi}_{t_\star}(\tau) \geq \frac{1}{\tau},
    \end{equation*}
    which is true by our choice of $t_\star$ and $\tau$; in particular, Claim \ref{clm: varphi increases for fixed t} gives
    \begin{equation*}\label{eq: cor6 (iii)}
        \tilde{\varphi}_{t_\star}(\tau) \geq \tilde{\varphi}_{t_\star}(\tau_\star)= 0.14368\ldots \geq 0.14367\ldots = \frac{1}{\tau_\star} \geq \frac{1}{\tau}.
    \end{equation*}
    
    Thus Proposition \ref{hky corollary 6} gives $\lp C -s\rp \cap C\cap \lp C+s\rp \neq \varnothing$ for all $s\in (0,0.435]$; hence, 
    $$G_{AP}^3(C)\supset (0,0.435].$$
    Scaling back by $d$ proves our theorem in full generality.
\end{proof}

\begin{proof}[Proof of Corollary \ref{cor:asymmetric patterns}]
    This follows from the proof above, though this time we apply Proposition \ref{hky corollary 6} to the sets
    $$C_1= C+s\theta, \quad C_2 = C-s(1-\theta), \quad \text{and}\quad C_3=C,$$
    where $s\in[0,0.87]$ and $\theta \in (0,1)$.

    The minor changes that follow are that $Q=[s\theta,1-s(1-\theta)]$ from which $|Q|=1-s$; this changes equation \eqref{eq:triple int cor6 (ii)} to 
    \begin{equation*}
        \tilde{\psi}_{t_\star}(\tau) \geq \tilde{\psi}_{t_\star}(\tau_\star) = 0.51954\ldots \geq 0.51538\ldots = \frac{|G_\star|}{1-s} \geq \frac{|G|}{1-s},
    \end{equation*}
    from which the doubled range of $s$ comes from.
\end{proof}

Lastly, we go one step beyond Theorem \ref{thm:hky triple intersection} to establish and quantify the existence of quadruple arithmetic progressions.

\begin{proof}[Proof of Theorem \ref{thm:hky quadruple intersection}]
    To apply Proposition \ref{prop:four-AP-HKY-general}, we need only find a set $C$ of sufficiently large thickness $\tau:=\tau(C)$ such that $$\tilde{\varphi}_t(\tau)\geq 1.$$
    We claim that any Cantor set $C\subset \R$ with normalized length $\conv(C)=[0,1]$ and thickness 
    $$\tau:=\tau(C) \geq \frac{1-0.015}{2(0.015)}= 32.83333\ldots$$
    is sufficient. 

    Consider the anchor thickness $\tau_\star := 32.8$ and $t$-value $t_\star:= 2.5$. 
    Observe that 
    $$0<t_\star<\frac{\tau_\star^2-2\tau_\star-1}{3\tau_\star+1} \leq\frac{\tau^2-2\tau-1}{3\tau+1},$$
    and
    \begin{align*}
         0.78197\ldots =\tilde{\psi}_{t_\star}(\tau_\star) \leq \tilde{\psi}_{t_\star}(\tau) \quad \text{ and }\quad  1.02129\ldots = \tilde{\varphi}_{t_\star}(\tau_\star) \leq \tilde{\varphi}_{t_\star}(\tau).
    \end{align*}
    Thus, Proposition \ref{prop:four-AP-HKY-general} applies and we conclude $$[0,0.3]\subset G_{AP}^4(C)$$ where the upper bound of the interval comes from 
    \begin{align*}
        0.30498\ldots = \frac{\tilde{\psi}_{t_\star}(\tau_\star)}{1+2\tilde{\psi}_{t_\star}(\tau_\star)} \leq \frac{\tilde{\psi}_{t_\star}(\tau)}{1+2\tilde{\psi}_{t_\star}(\tau)}.
    \end{align*}
\end{proof}

\section{Concluding remarks}

These results exhibit two complementary mechanisms for progression scale sets.
Self-similar first splitting creates recursively scaled blackout intervals,
while thickness and thick-intersection arguments create intervals of
admissible scales.  The level-two computation shows that product geometry can
improve a purely first-level estimate without requiring a full symbolic
description of the Cantor set.

Several natural problems remain.  The most immediate is to determine the
first-split spectrum $B_\eps$ exactly, or to compute a tractable level-three
outer approximation.  This could lower the blackout threshold below
$(4-\sqrt{13})/3$.  It would also be useful to optimize the numerical HKY
parameters and to develop analogous recursive descriptions for longer or
asymmetric patterns.

\appendix
\section{Endpoint certificate for the level-two calculation}
\label{app:level-two-certificate}

This appendix gives a direct check of the six rows used in Proposition~\ref{prop:level-two-first-split}.  
Write
$$
A=I_{00},\qquad B=I_{01},\qquad C=I_{10},\qquad D=I_{11},
$$
and identify each interval with its left endpoint
$$
0,\qquad \lambda-\lambda^2,\qquad 1-\lambda,\qquad 1-\lambda^2,
$$
respectively; all four intervals have length $L=\lambda^2$.
For a triple of left endpoints $(p,q,r)$, formula \eqref{eq:three-interval-feasibility} gives the lower candidates
$$
q-p-L,\qquad \frac{r-p-L}{2},\qquad r-q-L,\qquad0
$$
and the upper candidates
$$
q-p+L,\qquad \frac{r-p+L}{2},\qquad r-q+L.
$$
The following comparisons hold throughout $1/3\le\lambda<1/2$.

\begin{enumerate}[leftmargin=2.2em,label=\textup{(\roman*)}]
\item For $(A,A,C)$, the maximum lower candidate is $c_\lambda=1-\lambda-\lambda^2$ and the minimum upper candidate is $\lambda^2$.  
The interval is nonempty exactly when $c_\lambda\le\lambda^2$, equivalently $\lambda\ge1/2$.

\item For $(A,A,D)$, the endpoints are
$1-2\lambda^2$ and $\lambda^2$.
Nonemptiness is equivalent to $\lambda\ge1/\sqrt3$, so this row is empty in the range under study.

\item For $(A,B,C)$, the lower endpoint is $\max\{\eps,a_\lambda\}$ and the upper endpoint is $\min\{\lambda,b_\lambda\}$.  
The simultaneous equalities $a_\lambda=\eps$ and $b_\lambda=\lambda$ reduce to $\lambda^2-3\lambda+1=0$, whose root in $(0,1/2)$ is $\rho=(3-\sqrt5)/2$.  
Hence this row contributes $[\eps,\lambda]$ for $\lambda\le\rho$ and $[a_\lambda,b_\lambda]$ for $\lambda\ge\rho$.

\item For $(A,B,D)$, the maximum lower candidate is $c_\lambda$ and the minimum upper candidate is $\lambda$.  
Thus the contribution is $[c_\lambda,\lambda]$, nonempty exactly when $1-2\lambda-\lambda^2\le0$, or $\lambda\ge\eta=\sqrt2-1$.

\item For $(B,B,C)$, the maximum lower candidate is $\eps=1-2\lambda$ and the minimum upper candidate is $\lambda^2$.
Hence the contribution is $[\eps,\lambda^2]$, again nonempty exactly when $\lambda\ge\eta$.

\item For $(B,B,D)$, the calculation is the same as in (i), and there is no contribution for $\lambda<1/2$.
\end{enumerate}

It remains only to merge the three possible nonempty intervals.  
For $\rho<\lambda<\eta$, only $[a_\lambda,b_\lambda]$ remains.  
For $\eta\le\lambda<\kappa$, their order is
$$
\eps\le\lambda^2<a_\lambda\le b_\lambda<c_\lambda\le\lambda.
$$
The two strict gaps close simultaneously when
$$
\lambda^2=a_\lambda
\quad\Longleftrightarrow\quad
b_\lambda=c_\lambda
\quad\Longleftrightarrow\quad
3\lambda^2+\lambda-1=0,
$$
namely at $\kappa=(\sqrt{13}-1)/6$.  
For $\lambda\ge\kappa$ the three intervals overlap and fill $[\eps,\lambda]$.  
This independently reproduces all four cases of Proposition~\ref{prop:level-two-first-split}.

\section*{Acknowledgements}
A.~Y. is supported in part by the Natural Sciences and Engineering Research Council of Canada, NSERC (GR030571 and GR030540).
S.S. is supported in part by NSF DMS-2231565
K.T. is supported in part by the Simons Foundation Grant GR137264.

\end{document}